\documentclass[12pt]{article}
\usepackage{amsfonts}
\usepackage{amssymb}
\usepackage{amsmath,lineno}
\usepackage[numbers,sort&compress]{natbib}
\usepackage[colorlinks,linkcolor=blue,citecolor=blue,urlcolor=blue]{hyperref}
\usepackage{graphicx,indentfirst,tabularx}

\newtheorem{theorem}{Theorem}

\newtheorem{definition}[theorem]{Definition}
\newtheorem{example}[theorem]{Example}

\newtheorem{lemma}[theorem]{Lemma}

\newtheorem{remark}[theorem]{Remark}

\newenvironment{proof}[1][Proof]{\noindent\textbf{#1.} }{\ \hfill \ensuremath{\Box}}
\input{tcilatex}
\begin{document}

\title{The Multivariate Taylor Measure Function Space and a Generalization
of Taylor's Theorem}
\author{Athanasios C. Micheas \thanks{%
Department of Statistics, University of Missouri, 146 Middlebush Hall,
Columbia, MO 65211-6100, USA, email: micheasa@missouri.edu}}
\maketitle

\begin{abstract}
We propose a generalization of Taylor's theorem to measurable, non-analytic
functions, that do not require differentiation. We study consequences of the
generalization, including the definition and properties of a new space of
functions, which will be called the multivariate Taylor measure function
space. We illustrate through examples, that the proposed generalization
emerges as a unifying framework that includes many concepts from mathematics
as special cases. An application of the new function space is also presented
regarding solutions to linear and non-linear ordinary differential equations.
\end{abstract}



\section{Introduction}

One of the most fundamental methods of function approximation in mathematics
and related fields is Taylor's theorem, with many important applications
emerging over the years including numerical algorithms for optimization (%
\cite{More1978}; \cite{Conn2000}), state estimation (\cite{Sarkka2013}, Ch.
5), ordinary differential equations (\cite{Hairer1993}, Ch. 2, \cite%
{ramos2025piecewise}), Taylor expansions for vector valued functions (\cite%
{feng2014exact}), solutions of stochastic partial differential equations (%
\cite{jentzen2010taylor}), approximation in statistical distribution theory (%
\cite{micheas2018theory}, \cite{tsagris2014folded}), hitting probabilities
in Brownian motion (\cite{hobson1999taylor}) and approximation of
exponential integrals in Bayesian statistics (\cite{Raudenbush2000}).
Algorithms based on Taylor{'}s approximation that can be viewed as
statistical inference problems include spline interpolation (\cite%
{Diaconis1988}; \cite{Kimeldorf1970}), numerical quadrature (\cite%
{Diaconis1988}; \cite{Karvonen2017}; \cite{Karvonenetal2018}, differential
equations (\cite{Schober2014}; \cite{Schober2019}; \cite{Teymur2016}), and
linear algebra (\cite{Cockayneetal2019a}; \cite{Hennig2015}).

Motivated by Taylor's series, the concept of a Taylor measure was defined
and studied in \cite{micheas2025taylor}. Let $\mathcal{B}(\mathbb{N})$
denote the Borel sets of $\mathbb{N}=\{0,1,2,...,\},$ under the usual
topology, and denote the collection of all signed, finite Taylor measures by%
\begin{equation}
\mathcal{T}^{\mathcal{F}}=\left\{ T_{\gamma ,\mathbf{a}}:T_{\gamma ,\mathbf{a%
}}(B)=\sum\limits_{n\in B}a_{n}\frac{\gamma ^{n}}{n!}\text{, }B\in \mathcal{B%
}(\mathbb{N}),\text{ }a_{n},\gamma \in \Re ,\text{ with }T_{\gamma ,\mathbf{a%
}}(\mathbb{N})<+\infty \right\} .  \label{FiniteTaylorMeasureSpace}
\end{equation}%
Properties of the space $\mathcal{T}^{\mathcal{F}}$ and first applications
where presented in the latter paper, where a first connection was
illustrated between Taylor's theorem and $\mathcal{T}^{\mathcal{F}},$ for
univariate, real-valued analytic functions. More precisely, using the Taylor
series for an analytic function $f,$ we have%
\begin{equation}
f(x)=\sum_{n=0}^{+\infty }\frac{f^{(n)}(x_{0})}{n!}(x-x_{0})^{n}=%
\sum_{n=0}^{+\infty }a_{n}\frac{\gamma ^{n}(x)}{n!}=T_{\gamma (x),\mathbf{a}%
}(\mathbb{N}),  \label{UnivariateTM}
\end{equation}%
$\gamma (x)=x-x_{0},$ and $a_{n}=f^{(n)}(x_{0}),$ $n\in \mathbb{N},$ so that
any real-valued analytic function from $\Re $ is represented via a Taylor
measure. More generally, any analytic function $f$ is obtained by
introducing an input $x\in \Re $ in the parameter $\gamma $ of a finite
Taylor measure, i.e., we define $T_{\gamma (x),\mathbf{a}}$, for some
analytic function $\gamma :\Re \rightarrow \Re $. In the exposition that
follows we will mostly concentrate on the real case, however generalizations
to the complex plane are straightforward.

In this paper we propose a generalization of Taylor's theorem to measurable,
non-analytic functions, and study properties of the resulting space of
functions, based on the concept of the Taylor measure. We further propose
first applications in mathematics, and in particular functional analysis and
differential equations. The proposed space of functions emerges as a
unifying framework that contains many important concepts from mathematics
and functional analysis, such as orthogonal polynomials, hypergeometric
series and functions, Lie series, Riemann's $\zeta $ function, Rodrigues'
formula, ordinary differential equations, and more, and their modifications
and extensions in the literature. The resulting space of functions will be
aptly called the multivariate Taylor measure function space.

The paper proceeds as follows; in Section 2 we introduce the multivariate
Taylor measure function space, show that it is a generalization of Taylor's
theorem, study its properties extensively and illustrate that it contains a
plethora of important functions as special cases. We present a first
application of multivariate Taylor measure functions to differential
equations in Section 3. Concluding remarks are given in the last section.

\section{Taylor Measure Function Space}

In order to generalize Taylor's theorem and the representation of equation (%
\ref{UnivariateTM}) to higher dimensions, we require an alternative
definition of a function via a Taylor measure. We can think of the following
definition as a generalization of Taylor's Theorem, and as we will see, it
includes even non-analytic functions.

\begin{definition}[Multivariate Taylor Measure Function]
\label{MTMFDef}Consider a function $f^{B}:\Re ^{p}\rightarrow \Re $, defined
via the finite, signed Taylor measure%
\begin{equation}
f^{B}(\mathbf{x})=T_{g(\mathbf{x}),\mathbf{a}(\mathbf{x})}(B)=\sum\limits_{n%
\in B}a_{n}(\mathbf{x})\frac{g(\mathbf{x})^{n}}{n!}<+\infty ,
\label{MultTaylorFunction}
\end{equation}%
with $T_{g(\mathbf{x}),\mathbf{a}(\mathbf{x})}\in \mathcal{T}^{\mathcal{F}}$
(keeping $\mathbf{x}\in \Re ^{p}$ fixed), where $\mathbf{a}(\mathbf{x}%
)=[a_{0}(\mathbf{x}),$ $a_{1}(\mathbf{x}),...]\in \Re ^{\infty },$ and $%
a_{n},g:\Re ^{p}\rightarrow \Re ,$ are measurable functions in $(\Re ^{p},%
\mathcal{B}(\Re ^{p}))$, for all $B\in \mathcal{B}(\mathbb{N}),$ and $n\in
\mathbb{N}.$ The function $f^{B}$ will be called a multivariate Taylor
measure function (MTMF), and we use $\mathcal{F}_{\Re }^{\Re ^{p}}$ to
denote this collection of real-valued functions from $\Re ^{p}$.\newline
When $B$ is countably infinite, the sum of Equation (\ref{MultTaylorFunction}%
) converges under similar assumptions required for the Taylor measure to be
finite; keeping $\mathbf{x}\in \Re ^{p}$ fixed, $T_{g(\mathbf{x}),\mathbf{a}(%
\mathbf{x})}(B)$ is finite provided that one of the following conditions
hold:\newline
a) $a_{n}(\mathbf{x})$ are uniformly bounded, i.e., $|a_{n}(\mathbf{x})|\leq
M,$ $\forall n$, for some $M>0,$ or\newline
b) $a_{n}(\mathbf{x})$ are asymptotically equivalent to $Mb(\mathbf{x})^{n}$%
, i.e., $a_{n}(\mathbf{x})\thicksim Mb(\mathbf{x})^{n},$ for some measurable
function $b:\Re ^{p}\rightarrow \Re ,$ and $M\in \Re $, or\newline
c) if $a_{n}(\mathbf{x})\neq 0,$ $\forall \mathbf{x}\in \Re ^{p},$ then $%
\underset{n\rightarrow +\infty }{\lim }\frac{a_{n+1}(\mathbf{x})}{(n+1)a_{n}(%
\mathbf{x})}<1,$ guarantees absolute convergence (and therefore convergence).
\end{definition}

In the exposition that follows, we will assume that any MTMF under
consideration satisfies one of the three conditions a)-c) of Definition \ref%
{MTMFDef}, so that it is well defined.\ Moreover, since sums and products of
measurable functions are measurable, any member of $\mathcal{F}_{\Re }^{\Re
^{p}}$ is also a measurable function in $(\Re ^{p},\mathcal{B}(\Re ^{p})).$
Note that the above definition is a direct generalization of the
representation of equation (\ref{UnivariateTM}), since the latter is a
special case for an analytic $f^{B}$, with $a_{n}(x)=a_{n}=f^{(n)}(x_{0}),$ $%
g(x)=\gamma (x)=x-x_{0},$ $p=1,$ and $B=\mathbb{N}.$ Obviously, any given
measurable function $f:\Re ^{p}\rightarrow \Re $, is a special case of a
MTMF since for $B=\{1\}$, and $\mathbf{a}(\mathbf{x})=[0,1,0,...],$ we have%
\begin{equation}
f(\mathbf{x})=\sum\limits_{n\in B}a_{n}(\mathbf{x})\frac{f(\mathbf{x})^{n}}{%
n!}=T_{f(\mathbf{x}),\mathbf{a}(\mathbf{x})}(B).  \label{MTFsinglefunction}
\end{equation}%
We call (\ref{MTFsinglefunction}) the trivial MTMF representation of the
measurable function $f$. Therefore, the space $\mathcal{F}_{\Re }^{\Re ^{p}}$
represents a collection of measurable functions that contains any measurable
function, as well as, analytic functions as special cases. Next, we collect
the multivariate Taylor theorem as a special case of MTMFs.

\begin{example}[Multivariate Taylor Theorem as special case of MTMFs]
Recall that if $f:\Re ^{p}\rightarrow \Re $, is analytic at the point $%
\mathbf{x}_{0}=$ $[x_{1,0},x_{2,0},$ $\dots ,x_{p,0}]\in \Re ^{p},$ its
multivariate Taylor expansion is given by%
\begin{equation}
f(\mathbf{x})=\sum_{n_{1}=0}^{+\infty }\sum_{n_{2}=0}^{+\infty }\dots
\sum_{n_{p}=0}^{+\infty }\frac{a_{n_{1},n_{2},\dots ,n_{p}}}{%
n_{1}!n_{2}!\dots n_{p}!}(x_{1}-x_{1,0})^{n_{1}}(x_{2}-x_{2,0})^{n_{2}}\dots
(x_{p}-x_{p,0})^{n_{p}},  \label{MVTaylorThm}
\end{equation}%
where $\mathbf{x}=$ $[x_{1},x_{2},$ $\dots ,x_{p}]\in \Re ^{p},$ and%
\begin{equation*}
a_{n_{1},n_{2},\dots ,n_{p}}=\left( \frac{\partial ^{n_{1}+n_{2}+\dots
+n_{p}}f}{\partial x_{1}^{n_{1}}\partial x_{2}^{n_{2}}\dots \partial
x_{p}^{n_{p}}}\right) (\mathbf{x}_{0}),
\end{equation*}%
denotes all the coefficients. Since any multivariate analytic function is
analytic in each of its individual dimensions, we can rewrite (\ref%
{MVTaylorThm}) in terms of Taylor measures using%
\begin{eqnarray*}
f(\mathbf{x}) &=&\sum_{n_{1}=0}^{+\infty }\sum_{n_{2}=0}^{+\infty }\dots
\sum_{n_{p}=0}^{+\infty }\frac{a_{n_{1},n_{2},\dots ,n_{p}}}{%
n_{1}!n_{2}!\dots n_{p}!}(x_{1}-x_{1,0})^{n_{1}}(x_{2}-x_{2,0})^{n_{2}}\dots
(x_{p}-x_{p,0})^{n_{p}} \\
&=&\sum_{n_{1}=0}^{+\infty }\left[ \sum_{n_{2}=0}^{+\infty }\dots
\sum_{n_{p}=0}^{+\infty }a_{n_{1},n_{2},\dots ,n_{p}}\frac{%
(x_{2}-x_{2,0})^{n_{2}}\dots (x_{p}-x_{p,0})^{n_{p}}}{n_{2}!\dots n_{p}!}%
\right] \frac{(x_{1}-x_{1,0})^{n_{1}}}{n_{1}!} \\
&=&\sum\limits_{n_{1}\in \mathbb{N}}a_{n_{1},n_{2},\dots ,n_{p}}^{(-1)}(%
\mathbf{x}_{-1})\frac{(x_{1}-x_{1,0})^{n_{1}}}{n_{1}!}=T_{x_{1}-x_{1,0},%
\mathbf{a}_{\mathbf{n}}^{(-1)}}(\mathbb{N}),
\end{eqnarray*}%
where $\mathbf{n}=[n_{1},n_{2},\dots ,n_{p}]$, $\mathbf{x}%
_{-i}=[x_{i+1},...,x_{p}],$ $\mathbf{a}_{\mathbf{n}}^{(-1)}(\mathbf{x}_{-1})$
$=$ $[$ $a_{0,n_{2},...,n_{p}}^{(-1)}(\mathbf{x}_{-1}),$ $%
a_{1,n_{2},...,n_{p}}^{(-1)}(\mathbf{x}_{-1}),$ $...],$ and%
\begin{eqnarray*}
a_{n_{1},n_{2},...,n_{p}}^{(-i)}(\mathbf{x}_{-i})
&=&\sum_{n_{i+1}=0}^{+\infty }\dots \sum_{n_{p}=0}^{+\infty
}a_{n_{1},n_{2},\dots ,n_{p}}\frac{(x_{i+1}-x_{i+1,0})^{n_{i+1}}\dots
(x_{p}-x_{p,0})^{n_{p}}}{n_{i+1}!\dots n_{p}!} \\
&=&T_{x_{i}-x_{i,0},\mathbf{a}_{\mathbf{n}}^{(-(i+1))}}(\mathbb{N}),
\end{eqnarray*}%
$i=1,2,...,p-1$, so that%
\begin{equation*}
f(\mathbf{x})=T_{x_{1}-x_{1,0},\mathbf{a}_{\mathbf{n}}^{(-1)}}(\mathbb{N}%
)=T_{x_{1}-x_{1,0},T_{x_{2}-x_{2,0},\mathbf{a}_{\mathbf{n}}^{(-2)}}(\mathbb{N%
})}(\mathbb{N}).
\end{equation*}%
Letting $\mathbf{a}_{\mathbf{n}}^{(-p)}=[a_{n_{1},n_{2},\dots
,n_{p-1},0},a_{n_{1},n_{2},\dots ,n_{p-1},1},...],$ and proceeding as above
iteratively we have%
\begin{equation*}
f(\mathbf{x})=T_{x_{1}-x_{1,0},T_{x_{2}-x_{2,0},...,T_{x_{p}-x_{p,0},\mathbf{%
a}_{\mathbf{n}}^{(-p)}}(\mathbb{N})}...(\mathbb{N})}(\mathbb{N}),
\end{equation*}%
and as a result, the standard multivariate Taylor theorem can be expressed
in terms of MTMFs.
\end{example}

An important special case of MTMFs is discussed in the following.

\begin{example}[Simple and Measurable functions]
\label{MTMFEx}Let $B\in \mathcal{B}(\mathbb{N})$ and let $(\Re ^{p},\mathcal{%
B}(\Re ^{p}))$ denote the Borel measurable space in $\Re ^{p}$ (using the
usual Euclidean topology)$.$ Consider a partition $\{C_{n}\}_{n\in B}$ of $%
\Re ^{p},$ with $C_{n}\in \mathcal{B}(\Re ^{p}),$ and some real constants $%
c_{n}$. Define $g(\mathbf{x})=1,$ and $\mathbf{a}(\mathbf{x})$ with $a_{n}(%
\mathbf{x})=n!c_{n}I_{C_{n}}(\mathbf{x}),$ $n\in B,$ and $a_{n}(\mathbf{x}%
)=0,$ $n\notin B.$ Using equation (\ref{MultTaylorFunction}), we have that
any (canonical) simple function%
\begin{equation*}
f^{B}(\mathbf{x})=T_{g(\mathbf{x}),\mathbf{a}(\mathbf{x})}(B)=\sum\limits_{n%
\in B}a_{n}(\mathbf{x})\frac{g(\mathbf{x})^{n}}{n!}=\sum\limits_{n\in
B}c_{n}I_{C_{n}}(\mathbf{x})<+\infty ,
\end{equation*}%
is a member of $\mathcal{F}_{\Re }^{\Re ^{p}}.$ The importance of this
representation cannot be overstated, since any measurable function can be
expressed as a limit of a monotone sequence of simple measurable functions
from $\mathcal{F}_{\Re }^{\Re ^{p}}.$ Therefore, the space $\mathcal{F}_{\Re
}^{\Re ^{p}}$ contains both the collection of measurable and analytic
real-valued functions from $\Re ^{p},$ as well as limits of sequences of
such functions, i.e., $\mathcal{F}_{\Re }^{\Re ^{p}}$ is a closed set of
functions under pointwise limits.
\end{example}

Before we investigate general properties of the space of functions $\mathcal{%
F}_{\Re }^{\Re ^{p}},$ we prove a useful result for multivariate,
real-valued functions.

\begin{lemma}
\label{ConjFuncs}For any measurable functions $a_{n},b,c_{n},d:\Re
^{p}\rightarrow \Re $, $n\in \mathbb{N},$ there exist measurable functions $%
g_{n},g:\Re ^{p}\rightarrow \Re $, $n\in \mathbb{N},$ such that%
\begin{equation}
a_{n}(\mathbf{x})b(\mathbf{x})^{n}+c_{n}(\mathbf{x})d(\mathbf{x})^{n}=g_{n}(%
\mathbf{x})g(\mathbf{x})^{n},  \label{Conjecture4Functions}
\end{equation}%
for all $\mathbf{x}\in \Re ^{p}$ and $n\in \mathbb{N}.$ More generally, for
any finite collection of measurable functions $a_{n,i},b_{i}:\Re
^{p}\rightarrow \Re $, $i=1,2,...,L<+\infty ,$ there exist measurable
functions $s_{n},s:\Re ^{p}\rightarrow \Re $, $n\in \mathbb{N},$ such that%
\begin{equation}
\sum\limits_{i=1}^{L}a_{n,i}(\mathbf{x})\left( b_{i}(\mathbf{x})\right)
^{n}=s_{n}(\mathbf{x})s(\mathbf{x})^{n}.  \label{ConjLFuncs}
\end{equation}
\end{lemma}

\begin{proof}
Take arbitrary $\mathbf{x}\in \Re ^{p}$. For $n=1$, we have a real
measurable function $u(\mathbf{x})=a_{1}(\mathbf{x})b(\mathbf{x})+c_{1}(%
\mathbf{x})d(\mathbf{x}),$ and there are infinite selections of functions $%
g_{1}$ and $g$ that satisfy (\ref{Conjecture4Functions}), with $u(\mathbf{x}%
)=g_{1}(\mathbf{x})g(\mathbf{x})$. Assume that (\ref{Conjecture4Functions})
holds for $n\in \mathbb{N}$. We prove it holds for $n+1.$ We have%
\begin{eqnarray*}
a_{n+1}(\mathbf{x})b(\mathbf{x})^{n+1}+c_{n+1}(\mathbf{x})d(\mathbf{x}%
)^{n+1} &=&(a_{n+1}(\mathbf{x})b(\mathbf{x}))b(\mathbf{x})^{n}+(c_{n+1}(%
\mathbf{x})d(\mathbf{x}))d(\mathbf{x})^{n} \\
&=&g_{n}(\mathbf{x})g(\mathbf{x})^{n}=\left( \frac{g_{n}(\mathbf{x})}{g(%
\mathbf{x})}\right) g(\mathbf{x})^{n+1},
\end{eqnarray*}%
so that by induction the claim holds. Clearly, the functions $g_{n}$ and $g$
are not unique. Note that if the left hand side in (\ref%
{Conjecture4Functions}) is non-zero we must have $g(\mathbf{x)}\neq 0,$ for
all $\mathbf{x}\in \Re ^{p},$ and if the left hand side is zero we can
simply choose $g_{n}=0,$ $\forall n\in \mathbb{N},$ and any function $g$.%
\newline
For the generalization, since we have the result for $L=2$, using induction
on $L$, we have immediately (\ref{ConjLFuncs}), for any finite $L$.
\end{proof}

\subsection{Properties of $\mathcal{F}_{\Re }^{\Re ^{p}}$}

In order to study $\mathcal{F}_{\Re }^{\Re ^{p}}$ we require an inner
product in the space that will allow us to investigate its properties.
Unfortunately, for any $f_{1},f_{2}\in \mathcal{F}_{\Re }^{\Re ^{p}},$ the
usual inner product%
\begin{equation}
\rho _{2}(f_{1},f_{2})=\int\limits_{\Re ^{p}}f_{1}(\mathbf{x})f_{2}(\mathbf{x%
})\mu _{p}(d\mathbf{x}).  \label{InnerProduct1}
\end{equation}%
where $\mu _{p}$ denotes Lebesgue measure in $(\Re ^{p},\mathcal{B}(\Re
^{p}))$, will not work in the space $\mathcal{F}_{\Re }^{\Re ^{p}},$ so we
define a new map that will serve as the inner product we need. Although any
measure can be used to define the integrals that follow, we will assume
without loss of generality Lebesgue measure as the integrating measure.
Special care is required in certain places since the constants used to
define elements of $\mathcal{T}^{\mathcal{F}}$ are now measurable functions
of $\mathbf{x}\in \Re ^{p}$.

\begin{lemma}
\label{InnerProductMTMELemma}Let $f_{1}^{B},f_{2}^{B}\in \mathcal{F}_{\Re
}^{\Re ^{p}}$, so that $f_{i}^{B}(\mathbf{x})=T_{g_{i}(\mathbf{x}),\mathbf{a}%
_{i}(\mathbf{x})}(B)<+\infty ,$ with $\mathbf{a}_{i}(\mathbf{x})=[a_{0,i}(%
\mathbf{x}),$ $a_{1,i}(\mathbf{x}),...],$ for some measurable functions $%
a_{n,i},g_{i}:\Re ^{p}\rightarrow \Re ,$ $i=1,2$, and define the map $\rho _{%
\mathcal{F}}:\mathcal{F}_{\Re }^{\Re ^{p}}\times \mathcal{F}_{\Re }^{\Re
^{p}}\rightarrow \Re ,$ by%
\begin{equation}
\rho _{\mathcal{F}}\left( f_{1}^{B},f_{2}^{B}\right) =\sum\limits_{n\in
B}\int\limits_{\Re ^{p}}a_{n,1}(\mathbf{x})a_{n,2}(\mathbf{x})\frac{(g_{1}(%
\mathbf{x})g_{2}(\mathbf{x}))^{n}}{n!}\mu _{p}(d\mathbf{x}),
\label{InnerProductMTMF}
\end{equation}%
for any $B\in \mathcal{B}(\mathbb{N})$, under appropriate conditions on the
functions $a_{n,i}$ and $g_{i},$ for $\rho _{\mathcal{F}}$ to be well
defined (see Remark \ref{CondRemarks} below for details). Then, equipped
with $\rho _{\mathcal{F}},$ the space $\mathcal{F}_{\Re }^{\Re ^{p}}$ is an
inner product space.
\end{lemma}

\begin{proof}
Let $f_{1}^{B},f_{2}^{B},f_{3}^{B}\in \mathcal{F}_{\Re }^{\Re ^{p}},$ and
take arbitrary $B\in \mathcal{B}(\mathbb{N}).$ First note that $\rho _{%
\mathcal{F}}$ is symmetric since%
\begin{eqnarray*}
\rho _{\mathcal{F}}\left( f_{1}^{B},f_{2}^{B}\right) &=&\sum\limits_{n\in
B}\int\limits_{\Re ^{p}}a_{n,1}(\mathbf{x})a_{n,2}(\mathbf{x})\frac{(g_{1}(%
\mathbf{x})g_{2}(\mathbf{x}))^{n}}{n!}\mu _{p}(d\mathbf{x}) \\
&=&\sum\limits_{n\in B}\int\limits_{\Re ^{p}}a_{n,2}(\mathbf{x})a_{n,1}(%
\mathbf{x})\frac{(g_{2}(\mathbf{x})g_{1}(\mathbf{x}))^{n}}{n!}\mu _{p}(d%
\mathbf{x}) \\
&=&\rho _{\mathcal{F}}\left( f_{2}^{B},f_{1}^{B}\right) ,
\end{eqnarray*}%
and positive definite for non-zero $f_{1}^{B},$ since%
\begin{equation*}
\rho _{\mathcal{F}}\left( f_{1}^{B},f_{1}^{B}\right) =\sum\limits_{n\in
B}\int\limits_{\Re ^{p}}a_{n,1}^{2}(\mathbf{x})\frac{(g_{1}^{2}(\mathbf{x}%
))^{n}}{n!}\mu _{p}(d\mathbf{x})>0.
\end{equation*}%
In the case where $\rho _{\mathcal{F}}\left( f_{1}^{B},f_{1}^{B}\right) =0,$
since the latter is a sum of non-negative positive functions $a_{n,1}^{2}(%
\mathbf{x})(g_{1}^{2}(\mathbf{x}))^{n}$, then we must have $a_{n,1}(\mathbf{x%
})g_{1}(\mathbf{x})=0,$ and as a consequence $f_{1}^{B}=0$ almost everywhere
with respect to $\mu _{p}.$\newline
Moreover, for arbitrary $a,b\in \Re ,$ using (\ref{Conjecture4Functions}) we
write%
\begin{equation*}
aa_{n,1}(\mathbf{x})g_{1}(\mathbf{x})^{n}+ba_{n,2}(\mathbf{x})g_{2}(\mathbf{x%
})^{n}=g_{n}(\mathbf{x})g(\mathbf{x})^{n},
\end{equation*}%
for some measurable functions $g_{n}$ and $g$, so that%
\begin{equation}
af_{1}^{B}(\mathbf{x})+bf_{2}^{B}(\mathbf{x})=\sum\limits_{n\in B}(aa_{n,1}(%
\mathbf{x})g_{1}(\mathbf{x})^{n}+ba_{n,2}(\mathbf{x})g_{2}(\mathbf{x})^{n})%
\frac{1}{n!}=\sum\limits_{n\in B}g_{n}(\mathbf{x})\frac{g(\mathbf{x})^{n}}{n!%
}.  \label{LinearFormMTMF}
\end{equation}%
As a result, we have linearity in the first argument since%
\begin{eqnarray*}
&&a\rho _{\mathcal{F}}\left( f_{1}^{B},f_{3}^{B}\right) +b\rho _{\mathcal{M}%
}\left( f_{2}^{B},f_{3}^{B}\right) \\
&=&\sum\limits_{n\in B}\int\limits_{\Re ^{p}}\left( aa_{n,1}(\mathbf{x}%
)a_{n,3}(\mathbf{x})(g_{1}(\mathbf{x})g_{3}(\mathbf{x}))^{n}+ba_{n,2}(%
\mathbf{x})a_{n,3}(\mathbf{x})(g_{2}(\mathbf{x})g_{3}(\mathbf{x}%
))^{n}\right) \frac{1}{n!}\mu _{p}(d\mathbf{x}) \\
&=&\sum\limits_{n\in B}\int\limits_{\Re ^{p}}\left( aa_{n,1}(\mathbf{x}%
)g_{1}^{n}(\mathbf{x})+ba_{n,2}(\mathbf{x})g_{2}^{n}(\mathbf{x})\right)
a_{n,3}(\mathbf{x})\frac{g_{3}^{n}(\mathbf{x})}{n!}\mu _{p}(d\mathbf{x}) \\
&=&\sum\limits_{n\in B}\int\limits_{\Re ^{p}}\left( g_{n}(\mathbf{x})g(%
\mathbf{x})^{n}\right) a_{n,3}(\mathbf{x})\frac{g_{3}^{n}(\mathbf{x})}{n!}%
\mu _{p}(d\mathbf{x})=\sum\limits_{n\in B}\int\limits_{\Re ^{p}}g_{n}(%
\mathbf{x})a_{n,3}(\mathbf{x})\frac{\left( g(\mathbf{x})g_{3}(\mathbf{x}%
)\right) ^{n}}{n!}\mu _{p}(d\mathbf{x}) \\
&=&\rho _{\mathcal{F}}\left( af_{1}^{B}+bf_{2}^{B},f_{3}^{B}\right) ,
\end{eqnarray*}%
so that $\rho _{\mathcal{F}}(.,.)$ defines an inner product in $\mathcal{F}%
_{\Re }^{\Re ^{p}}$.
\end{proof}

Now based on the inner product $\rho _{\mathcal{F}},$ we can immediately
equip $\mathcal{F}_{\Re }^{\Re ^{p}}$ with the induced norm%
\begin{equation}
\left\Vert f^{B}\right\Vert _{\rho _{\mathcal{F}}}=\sqrt{\rho _{\mathcal{F}%
}\left( f^{B},f^{B}\right) }=\sqrt{\sum\limits_{n\in B}\int\limits_{\Re
^{p}}a_{n}(\mathbf{x})^{2}\frac{g(\mathbf{x})^{2n}}{n!}\mu _{p}(d\mathbf{x})}%
,  \label{MTMNorm}
\end{equation}%
and distance%
\begin{equation}
d_{\mathcal{F}}\left( f_{1}^{B},f_{2}^{B}\right) =\left\Vert
f_{1}^{B}-f_{2}^{B}\right\Vert _{\rho _{\mathcal{F}}}=\sqrt{\rho _{\mathcal{F%
}}\left( f_{1}^{B}-f_{2}^{B},f_{1}^{B}-f_{2}^{B}\right) },
\label{MTMDistance}
\end{equation}%
so that $\mathcal{F}_{\Re }^{\Re ^{p}}$ becomes a normed vector space and $%
\left( \mathcal{F}_{\Re }^{\Re ^{p}},d_{\mathcal{F}}\right) $ is a metric
space. Note here that for arbitrary $B_{1},B_{2}\in \mathcal{B}(\mathbb{N}),$
$B_{1}\neq B_{2},$ with $C=B_{1}\cap B_{2},$ we can still define the inner
product via%
\begin{equation}
\rho _{\mathcal{F}}\left( f_{1}^{B_{1}},f_{2}^{B_{2}}\right) =\rho _{%
\mathcal{F}}\left( f_{1}^{C},f_{2}^{C}\right) ,
\end{equation}%
for any $f_{1}^{B_{1}},f_{2}^{B_{2}}\in \mathcal{F}_{\Re }^{\Re ^{p}},$
where we set $a_{n,1},a_{n,2}=0,$ for $n\in \mathbb{N}\smallsetminus C.$

In view of equation (\ref{MTFsinglefunction}), we can easily verify that (%
\ref{InnerProduct1}) is a special case of (\ref{InnerProductMTMF}) for $%
B=\{1\},$ and $\mathbf{a}_{1}(x)=\mathbf{a}_{2}(x)=[0,1,0,...],$ $g_{1}(%
\mathbf{x})=f_{1}(\mathbf{x})$ and $g_{2}(\mathbf{x})=f_{2}(\mathbf{x}),$ so
that the standard inner product for functions in the literature, is a
special case of $\rho _{\mathcal{F}}$.

In order to rigorously justify how to swap the order of summation and
integration in the latter lemma, as well as the sums of the integrals in
Equation (\ref{InnerProductMTMF}) to be well defined and finite, we
summarize the additional conditions we will require next. One of these
conditions will be assumed to hold for all the proofs of the results that
follow in the exposition throughout the rest of the paper, in order to
guarantee the validity of the mathematical statements.

\begin{remark}[Integrability Conditions for $a_{n,i}$ and $g_{i}$]
\label{CondRemarks}In general, in order to swap the order of summation and
integration, and in addition to have a finite result, one can appeal to
Lebesgue's Dominated Convergence Theorem (DCT) or Fubini-Tonelli theorems
(see \cite{micheas2018theory}, Theorems 3.14, 3.22 and 3.23). Depending on
which of the conditions a)-b) of Definition \ref{MTMFDef} are satisfied in
order for the MTMFs $f_{i}^{B}(\mathbf{x})=T_{g_{i}(\mathbf{x}),\mathbf{a}%
_{i}(\mathbf{x})}(B)<+\infty ,$ to be well defined and finite, we will
impose corresponding integrability conditions in order for $\rho _{\mathcal{F%
}}\left( f_{1}^{B},f_{2}^{B}\right) $ to be finite and to be able to swap
the integration and summation signs.\newline
In view of the Cauchy-Schwarz inequality for Hilbert spaces we have%
\begin{equation}
\left\vert \rho _{\mathcal{F}}\left( f_{1}^{B},f_{2}^{B}\right) \right\vert
\leq \left\Vert f_{1}^{B}\right\Vert _{\rho _{\mathcal{F}}}\left\Vert
f_{2}^{B}\right\Vert _{\rho _{\mathcal{F}}},  \label{CauchySchwarz}
\end{equation}%
so that integrability conditions required for $\rho _{\mathcal{F}}\left(
f_{1}^{B},f_{2}^{B}\right) <+\infty ,$ can be obtained via integrability
conditions for $\left\Vert f_{i}^{B}\right\Vert _{\rho _{\mathcal{F}%
}}^{2}<+\infty ,$ $i=1,2.$

\begin{enumerate}
\item Finite sums: For the easy case, when $B\in \mathcal{B}(\mathbb{N}),$
is finite we do not have to worry about swapping summation and integration
signs and we only require%
\begin{equation}
\left\Vert f^{B}\right\Vert _{\rho _{\mathcal{F}}}^{2}=\sum\limits_{n\in
B}\int\limits_{\Re ^{p}}a_{n}(\mathbf{x})^{2}\frac{g(\mathbf{x})^{2n}}{n!}%
\mu _{p}(d\mathbf{x})<+\infty ,  \label{L2condintegrability}
\end{equation}%
for any $f^{B}\in \mathcal{F}_{\Re }^{\Re ^{p}}$. In terms of the underlying
functions $a_{n}$ and $g,$ in order for $\left\Vert f^{B}\right\Vert _{\rho
_{\mathcal{F}}}^{2}<+\infty ,$ we require%
\begin{equation}
\int\limits_{\Re ^{p}}a_{n}^{2}(\mathbf{x})g(\mathbf{x})^{2n}\mu _{p}(d%
\mathbf{x})<+\infty ,  \label{IntegrabilityCond1}
\end{equation}%
for all $n\in B.$ The latter condition is satisfied for measurable functions
$a_{n}$ and $g,$ under the following scenarios based on the MTMF $f^{B}(%
\mathbf{x})=T_{g(\mathbf{x}),\mathbf{a}(\mathbf{x})}(B)<+\infty ,$
conditions:

\begin{enumerate}
\item Uniformly bounded: When $|a_{n}(\mathbf{x})|\leq M,$ we will require%
\begin{equation}
\int\limits_{\Re ^{p}}g(\mathbf{x})^{n}\mu _{p}(d\mathbf{x})<+\infty ,
\label{CondIntegrab1}
\end{equation}%
for all $n\in \mathbb{N}.$

\item Asymptotically equivalent: When $a_{n}(\mathbf{x})\thicksim Mb(\mathbf{%
x})^{n},$ for some function $b:\Re ^{p}\rightarrow \Re ,$ and $M\in \Re $,
we require%
\begin{equation}
\int\limits_{\Re ^{p}}\left( b(\mathbf{x})g(\mathbf{x})\right) ^{n}\mu _{p}(d%
\mathbf{x})<+\infty ,  \label{CondIntegrab2}
\end{equation}%
for all $n\in \mathbb{N}.$
\end{enumerate}

\item Countably infinite sum: For the harder case, assume that $B$ is
countably infinite, and appeal to DCT. More precisely, for any sequence of
functions $f_{n}(\mathbf{x}),$ assume that there exists a sequence of
functions $h_{n}(\mathbf{x})$ such that
\begin{equation}
\left\vert f_{n}(\mathbf{x})\right\vert \leq h_{n}(\mathbf{x}),\forall n\in
\mathbb{N}\text{ and almost all }\mathbf{x}\in \Re ^{p},  \label{DCTCond1}
\end{equation}%
and in addition, assume that%
\begin{equation}
\sum\limits_{n\in B}h_{n}(\mathbf{x})<+\infty ,  \label{DCTCond2}
\end{equation}%
and%
\begin{equation}
\int\limits_{\Re ^{p}}h_{n}(\mathbf{x})\mu _{p}(d\mathbf{x})<+\infty .
\label{DCTCond3}
\end{equation}%
Then we can swap the order of summation and integration and write%
\begin{equation}
\sum\limits_{n\in B}\int\limits_{\Re ^{p}}f_{n}(\mathbf{x})\mu _{p}(d\mathbf{%
x})=\int\limits_{\Re ^{p}}\sum\limits_{n\in B}f_{n}(\mathbf{x})\mu _{p}(d%
\mathbf{x}).  \label{SwapSumIntegral}
\end{equation}%
In terms of $\rho _{\mathcal{F}}\left( f_{1}^{B},f_{2}^{B}\right) ,$ we have%
\begin{equation*}
\left\vert f_{n}(\mathbf{x})\right\vert =\left\vert a_{n,1}(\mathbf{x}%
)a_{n,2}(\mathbf{x})\right\vert \frac{\left\vert g_{1}(\mathbf{x})g_{2}(%
\mathbf{x})\right\vert ^{n}}{n!},
\end{equation*}%
which will require the following conditions;

\begin{enumerate}
\item Uniformly bounded: Under condition a) of Definition \ref{MTMFDef},
when $|a_{n,1}(\mathbf{x})|\leq M_{1},$ $|a_{n,2}(\mathbf{x})|\leq M_{2},$
for all $n\in \mathbb{N},$ yields%
\begin{equation*}
\left\vert f_{n}(\mathbf{x})\right\vert \leq h_{n}(\mathbf{x})=\frac{%
M_{1}M_{2}}{n!}\left\vert g_{1}(\mathbf{x})g_{2}(\mathbf{x})\right\vert ^{n},
\end{equation*}%
with condition (\ref{DCTCond2}) satisfied since%
\begin{equation*}
\sum\limits_{n\in B}\frac{M_{1}M_{2}}{n!}\left\vert g_{1}(\mathbf{x})g_{2}(%
\mathbf{x})\right\vert ^{n}\leq M_{1}M_{2}\sum\limits_{n\in \mathbb{N}}\frac{%
1}{n!}\left\vert g_{1}(\mathbf{x})g_{2}(\mathbf{x})\right\vert
^{n}=M_{1}M_{2}e^{\left\vert g_{1}(\mathbf{x})g_{2}(\mathbf{x})\right\vert
}<+\infty ,
\end{equation*}%
and condition (\ref{DCTCond3}) is also satisfied provided that%
\begin{equation}
\int\limits_{\Re ^{p}}\left\vert g_{1}(\mathbf{x})g_{2}(\mathbf{x}%
)\right\vert ^{n}\mu _{p}(d\mathbf{x})<+\infty .
\end{equation}

\item Asymptotically equivalent: On the other hand, under condition b) of
Definition \ref{MTMFDef}, we have $a_{n,1}(\mathbf{x})\thicksim M_{1}b_{1}(%
\mathbf{x})^{n},$ and $a_{n,2}(\mathbf{x})\thicksim M_{2}b_{2}(\mathbf{x}%
)^{n},$ for some measurable functions $b_{1},b_{2}:\Re ^{p}\rightarrow \Re ,$
and $M_{1},M_{2}\in \Re $, we have that $\forall \varepsilon >0,$ $\exists
n_{0}>0,$ such that for all $n>n_{0},$ we have
\begin{equation*}
\left\vert \frac{a_{n,1}(\mathbf{x})}{M_{1}b_{1}(\mathbf{x})^{n}}%
-1\right\vert <\varepsilon \Rightarrow a_{n,1}(\mathbf{x})<M_{1}b_{1}(%
\mathbf{x})^{n}(\varepsilon +1),
\end{equation*}%
and similarly, $\exists n_{1}>0,$ such that for all $n>n_{1},$ yields%
\begin{equation*}
\left\vert \frac{a_{n,2}(\mathbf{x})}{M_{2}b_{2}(\mathbf{x})^{n}}%
-1\right\vert <\varepsilon \Rightarrow a_{n,2}(\mathbf{x})<M_{2}b_{2}(%
\mathbf{x})^{n}(\varepsilon +1).
\end{equation*}%
Then we can write%
\begin{equation*}
\left\vert f_{n}(\mathbf{x})\right\vert \leq h_{n}(\mathbf{x})=\frac{%
\left\vert M_{1}M_{2}\right\vert }{n!}(\varepsilon +1)^{2}\left\vert b_{1}(%
\mathbf{x})b_{2}(\mathbf{x})g_{1}(\mathbf{x})g_{2}(\mathbf{x})\right\vert
^{n},
\end{equation*}%
and sending $\varepsilon \rightarrow 0,$%
\begin{equation*}
\left\vert f_{n}(\mathbf{x})\right\vert \leq h_{n}(\mathbf{x})=\frac{%
\left\vert M_{1}M_{2}\right\vert }{n!}\left\vert b_{1}(\mathbf{x})b_{2}(%
\mathbf{x})g_{1}(\mathbf{x})g_{2}(\mathbf{x})\right\vert ^{n}.
\end{equation*}%
Consequently, condition (\ref{DCTCond2}) is satisfied since%
\begin{eqnarray*}
\sum\limits_{n\in B}\frac{\left\vert M_{1}M_{2}\right\vert }{n!}\left\vert
b_{1}(\mathbf{x})b_{2}(\mathbf{x})g_{1}(\mathbf{x})g_{2}(\mathbf{x}%
)\right\vert ^{n} &\leq &\left\vert M_{1}M_{2}\right\vert \sum\limits_{n\in
\mathbb{N}}\frac{1}{n!}\left\vert b_{1}(\mathbf{x})b_{2}(\mathbf{x})g_{1}(%
\mathbf{x})g_{2}(\mathbf{x})\right\vert ^{n} \\
&=&\left\vert M_{1}M_{2}\right\vert e^{\left\vert b_{1}(\mathbf{x})b_{2}(%
\mathbf{x})g_{1}(\mathbf{x})g_{2}(\mathbf{x})\right\vert }<+\infty ,
\end{eqnarray*}%
and condition (\ref{DCTCond3}) is also satisfied provided that%
\begin{equation}
\int\limits_{\Re ^{p}}\left\vert b_{1}(\mathbf{x})b_{2}(\mathbf{x})g_{1}(%
\mathbf{x})g_{2}(\mathbf{x})\right\vert ^{n}\mu _{p}(d\mathbf{x})<+\infty .
\end{equation}%
for all $n\in \mathbb{N}.$
\end{enumerate}
\end{enumerate}
\end{remark}

We prove completeness of $\mathcal{F}_{\Re }^{\Re ^{p}}$ next.

\begin{lemma}
The space $\mathcal{F}_{\Re }^{\Re ^{p}}$ equipped with the norm $\left\Vert
.\right\Vert _{\rho _{\mathcal{F}}}$ is complete.
\end{lemma}

\begin{proof}
The proof follows similar steps as in the standard proof of completeness for
$\mathcal{L}^{p}$-spaces (see \cite{Dudley2004}, Theorem 5.2.1). Consider a
sequence of MTMFs $v_{n}=f_{n}^{B}\in \mathcal{F}_{\Re }^{\Re ^{p}},$ for $%
B\in \mathcal{B}(\mathbb{N}),$ with%
\begin{equation*}
v_{n}(\mathbf{x})=\sum\limits_{k\in B}a_{k,n}(\mathbf{x})\frac{g_{n}(\mathbf{%
x})^{k}}{k!},
\end{equation*}%
for some measurable functions $a_{k,n}(\mathbf{x})$ and $g_{n}(\mathbf{x}),$
and assume that $v_{n}(\mathbf{x})$ is Cauchy in $\mathcal{F}_{\Re }^{\Re
^{p}}$ with respect to $\left\Vert .\right\Vert _{\rho _{\mathcal{F}}}$,
i.e., $\forall \varepsilon >0,$ $\exists N>0,$ such that, $\forall m,n>N,$
we have%
\begin{equation*}
\left\Vert v_{m}-v_{n}\right\Vert _{\rho _{\mathcal{F}}}<\varepsilon .
\end{equation*}%
Because in any metric space a Cauchy sequence with a convergent subsequence
is convergent to the same limit, it is enough to prove convergence of a
subsequence. Therefore, we can assume that%
\begin{equation*}
\left\Vert v_{m}-v_{n}\right\Vert _{\rho _{\mathcal{F}}}<\frac{1}{2^{n}},
\end{equation*}%
for all $n\in \mathbb{N}$ and $m>n,$ where%
\begin{equation*}
v_{m}-v_{n}=\sum\limits_{k\in B}\int\limits_{\Re ^{p}}\frac{1}{k!}\left[
a_{k,m}(\mathbf{x})g_{m}(\mathbf{x})^{k}-a_{k,n}(\mathbf{x})g_{n}(\mathbf{x}%
)^{k}\right] \mu _{p}(d\mathbf{x}),
\end{equation*}%
and using (\ref{Conjecture4Functions}) we write%
\begin{equation*}
a_{k,m}(\mathbf{x})g_{m}(\mathbf{x})^{k}-a_{k,n}(\mathbf{x})g_{n}(\mathbf{x}%
)^{k}=c_{k,m,n}(\mathbf{x})g_{m,n}(\mathbf{x})^{k},
\end{equation*}%
for some measurable functions $c_{k,m,n}$ and $g_{m,n}$, which leads to%
\begin{equation*}
\left\Vert v_{m}-v_{n}\right\Vert _{\rho _{\mathcal{F}}}^{2}=\int\limits_{%
\Re ^{p}}\sum\limits_{k\in B}c_{k,m,n}(\mathbf{x})^{2}\frac{g_{m,n}(\mathbf{x%
})^{2k}}{k!}\mu _{p}(d\mathbf{x})<\frac{1}{2^{n}}.
\end{equation*}%
Let%
\begin{equation*}
A_{n}=\left\{ \mathbf{x}:\sum\limits_{k\in B}c_{k,n+1,n}(\mathbf{x})^{2}%
\frac{g_{n+1,n}(\mathbf{x})^{2k}}{k!}\geq 1/n^{2}\right\} ,
\end{equation*}%
which yields the inequality%
\begin{equation*}
\chi _{A_{n}}(\mathbf{x})/n^{2}\leq \sum\limits_{k\in B}c_{k,n+1,n}(\mathbf{x%
})^{2}\frac{g_{n+1,n}(\mathbf{x})^{2k}}{k!},
\end{equation*}%
where $\chi _{A_{n}}(\mathbf{x})$ the characteristic function of $A_{n}$,
and integrating the latter with respect to $\mu _{p}$ gives%
\begin{equation*}
\mu _{p}(A_{n})/n^{2}\leq \int\limits_{\Re ^{p}}\sum\limits_{k\in
B}c_{k,n+1,n}(\mathbf{x})^{2}\frac{g_{n+1,n}(\mathbf{x})^{2k}}{k!}\mu _{p}(d%
\mathbf{x})<\frac{1}{2^{n}},
\end{equation*}%
for all $n\in \mathbb{N},$ with%
\begin{equation*}
\sum\limits_{n=0}^{+\infty }\mu _{p}(A_{n})\leq \sum\limits_{n=0}^{+\infty }%
\frac{n^{2}}{2^{n}}<+\infty .
\end{equation*}%
Now define $B_{n}=\bigcup\limits_{m\geq n}A_{m},$ so that $B_{n}$ is a
decreasing sequence with $\mu _{p}(B_{n})\rightarrow 0,$ as $n\rightarrow
\infty .$ For any $\mathbf{x}\notin \bigcap\limits_{n=1}^{+\infty }B_{n},$
and for almost all $\mathbf{x}$ and for all large enough $n$, we have%
\begin{equation*}
\sum\limits_{k\in B}c_{k,n+1,n}(\mathbf{x})^{2}\frac{g_{n+1,n}(\mathbf{x}%
)^{2k}}{k!}\mu _{p}(d\mathbf{x})\leq \frac{1}{n^{2}},
\end{equation*}%
and consequently%
\begin{equation*}
\sum\limits_{k\in B}c_{k,m,n}(\mathbf{x})^{2}\frac{g_{m,n}(\mathbf{x})^{2k}}{%
k!}\mu _{p}(d\mathbf{x})\leq \sum\limits_{j=n}^{+\infty }\frac{1}{j^{2}}.
\end{equation*}
for any $m>n.$ Since $\sum\limits_{j=1}^{+\infty }1/j^{2}$ converges, $%
\sum\limits_{j=n}^{+\infty }1/j^{2}\rightarrow 0,$ as $n\rightarrow \infty .$
Thus for such $\mathbf{x},$ $\{v_{n}(\mathbf{x})\}$ is a Cauchy sequence
that is convergent pointwise, with some limit, say $v(\mathbf{x})$. For all
other $\mathbf{x}$, forming a set of measure $0$, let $v(\mathbf{x})=0$.
Therefore $v$ is a measurable function, so that it assumes, at the very
least, a trivial MTMF representation, i.e., $v\in \mathcal{F}_{\Re }^{\Re
^{p}}$. Because a Cauchy sequence is bounded, by Fatou's Lemma (see \cite%
{Dudley2004}, Lemma 4.3.3), we have%
\begin{eqnarray*}
\left\Vert v\right\Vert _{\rho _{\mathcal{F}}}^{2} &=&\left\Vert \underset{%
n\rightarrow \infty }{\lim \inf }v_{n}\right\Vert _{\rho _{\mathcal{F}%
}}^{2}=\int\limits_{\Re ^{p}}\underset{n\rightarrow \infty }{\lim \inf }%
\left( \sum\limits_{k\in B}a_{k,n}(\mathbf{x})^{2}\frac{g_{n}(\mathbf{x}%
)^{2k}}{k!}\right) \mu _{p}(d\mathbf{x}) \\
&\leq &\underset{n\rightarrow \infty }{\lim \inf }\int\limits_{\Re
^{p}}\sum\limits_{k\in B}a_{k,n}(\mathbf{x})^{2}\frac{g_{n}(\mathbf{x})^{2k}%
}{k!}\mu _{p}(d\mathbf{x})=\underset{n\rightarrow \infty }{\lim \inf }%
\left\Vert v_{n}\right\Vert _{\rho _{\mathcal{F}}}^{2}<+\infty ,
\end{eqnarray*}%
so that $\left\Vert v\right\Vert _{\rho _{\mathcal{F}}}<+\infty .$ Finally,
to show that $\left\Vert v_{n}-v\right\Vert _{\rho _{\mathcal{F}%
}}\rightarrow 0,$ we write%
\begin{eqnarray*}
\left\Vert v_{n}-v\right\Vert _{\rho _{\mathcal{F}}}^{2} &\leq &\underset{%
m\rightarrow \infty }{\lim \inf }\left\Vert v_{n}-v_{m}\right\Vert _{\rho _{%
\mathcal{F}}}^{2} \\
&\leq &\underset{m\rightarrow \infty }{\lim \inf }\int\limits_{\Re
^{p}}\sum\limits_{k\in B}c_{k,m,n}(\mathbf{x})^{2}\frac{g_{m,n}(\mathbf{x}%
)^{2k}}{k!}\mu _{p}(d\mathbf{x})\rightarrow 0,
\end{eqnarray*}
as $n\rightarrow \infty ,$ as desired.
\end{proof}

Combining all the results above gives the following.

\begin{theorem}[Hilbert space]
The space of multivariate Taylor measure functions $\mathcal{F}_{\Re }^{\Re
^{p}},$ equipped with the inner product $\rho _{\mathcal{F}}\left(
.,.\right) ,$ is a Hilbert space.
\end{theorem}

Since $\mathcal{F}_{\Re }^{\Re ^{p}}$ is Hilbert, we can immediately obtain
results based on all the general theory on Hilbert spaces in the literature.
For example, $\mathcal{F}_{\Re }^{\Re ^{p}}$ equipped with the norm $%
\left\Vert .\right\Vert _{\rho _{\mathcal{F}}}$ is a closed, Banach space,
as well as a Hausdorff space, and following \cite{Dudley2004}, Theorems
5.4.7, 5.4.9 and Corollary 5.4.10, every Hilbert space has an orthonormal
basis. In particular, let $\mathcal{E}=\{e_{i}\}_{i\in I}$ denote an
orthonormal basis of $\mathcal{F}_{\Re }^{\Re ^{p}},$ where the index set $I$
is not necessarily countable, so that for any $f^{B}\in \mathcal{F}_{\Re
}^{\Re ^{p}},$ we have the reproducing formula%
\begin{equation}
f^{B}(\mathbf{x})=\sum_{i\in I}\rho _{\mathcal{F}}\left( f^{B},e_{i}\right)
e_{i}(\mathbf{x}),  \label{ReproducingHilbertTmeasure}
\end{equation}%
for any $B\in \mathcal{B}(\mathbb{N}).$ Clearly, since $\mathcal{E}$ is not
unique, the latter representation of the MTMF $f^{B}$ is not unique.

The following lemma provides additional insight to the topological structure
of the space $\mathcal{F}_{\Re }^{\Re ^{p}}.$

\begin{lemma}
\label{AlgebraFuncs}The space $\mathcal{F}_{\Re }^{\Re ^{p}}$\ is an algebra
of functions that includes constant functions and separates points, i.e., it
is a vector space of functions that is also closed under pointwise
multiplication, and for each $\mathbf{x}\neq \mathbf{y}$ there is a function
$f^{B}\in \mathcal{F}_{\Re }^{\Re ^{p}},$ with $f^{B}(\mathbf{x})\neq f^{B}(%
\mathbf{y})$, for some $B\in \mathcal{B}(\mathbb{N}).$
\end{lemma}

\begin{proof}
Let $f_{1}^{B},f_{2}^{B}\in \mathcal{F}_{\Re }^{\Re ^{p}},$ for some $B\in
\mathcal{B}(\mathbb{N}),$ with%
\begin{equation*}
f_{k}^{B}(\mathbf{x})=\sum\limits_{n\in B}a_{n,k}(\mathbf{x})\frac{g_{k}(%
\mathbf{x})^{n}}{n!},
\end{equation*}%
for measurable functions $a_{n,k},$ $g_{k},$ $k=1,2$, $n\in \mathbb{N}.$
Since $\mathcal{F}_{\Re }^{\Re ^{p}}$ is a Hilbert space, it is a vector
space (see equation (\ref{LinearFormMTMF})). Moreover, we can write%
\begin{equation*}
f_{1}^{B}(\mathbf{x})f_{2}^{B}(\mathbf{x})=\left( \sum\limits_{n\in
B}a_{n,1}(\mathbf{x})\frac{g_{1}(\mathbf{x})^{n}}{n!}\right) f_{2}^{B}(%
\mathbf{x})=\sum\limits_{n\in B}\left( a_{n,1}(\mathbf{x})f_{2}^{B}(\mathbf{x%
})\right) \frac{g_{1}(\mathbf{x})^{n}}{n!}\in \mathcal{F}_{\Re }^{\Re ^{p}},
\end{equation*}%
for any $\mathbf{x}\in \Re ^{p}$, where $a_{n,1}(\mathbf{x})f_{2}^{B}(%
\mathbf{x})$ is a measurable function, so that $\mathcal{F}_{\Re }^{\Re
^{p}} $ is closed under pointwise multiplication. Clearly, for any constant $%
c\in \Re ,$ taking $B=\{1\},$ $g(\mathbf{x})=1,$ and $\mathbf{a}(\mathbf{x}%
)=[0,c,0,...],$ we have $c=T_{g(\mathbf{x}),\mathbf{a}(\mathbf{x})}(B)\in
\mathcal{F}_{\Re }^{\Re ^{p}}.$ Now take arbitrary $\mathbf{x},\mathbf{y}\in
\Re ^{p}$, with $\mathbf{x}\neq \mathbf{y},$ and choose any 1 to 1
measurable function $g,$ take $B=\{1\},$ and $\mathbf{a}(\mathbf{x}%
)=[0,1,0,...],$ so that $g(\mathbf{x})=T_{g(\mathbf{x}),\mathbf{a}(\mathbf{x}%
)}(B)\in \mathcal{F}_{\Re }^{\Re ^{p}}$, is such that $g(\mathbf{x})\neq g(%
\mathbf{y}),$ and therefore, $\mathcal{F}_{\Re }^{\Re ^{p}}$ separates
points.
\end{proof}

Now let $\mathcal{C}_{\Re }^{\Re ^{p}}$ denote all continuous, multivariate,
real-valued functions, and $\mathcal{M}_{\Re }^{\Re ^{p}}$ the collection of
all multivariate, real-valued, measurable functions. In general, a
measurable function in $\mathcal{M}_{\Re }^{\Re ^{p}}$ is not necessarily
continuous (e.g., indicator functions of measurable rectangles), however any
continuous function is measurable, i.e., $\mathcal{C}_{\Re }^{\Re
^{p}}\subset \mathcal{M}_{\Re }^{\Re ^{p}}\subset \mathcal{F}_{\Re }^{\Re
^{p}}.$ Finding dense subsets in spaces of functions is a crucial, well
known and studied problem for spaces of continuous functions. For example,
application of classic theorems from functional analysis such as the
Stone-Weierstrass theorem (see \cite{Dudley2004}, Theorem 2.4.11), allows us
to verify if a subset is dense, using $d_{sup}(f,g)=\underset{\mathbf{x}\in K%
}{\sup }|f(\mathbf{x})-g(\mathbf{x})|$ norm, but they require a compact
space $K\subset \Re ^{p},$ i.e., the function space $\mathcal{C}_{\Re }^{K}$%
, or apply modifications of Stone-Weierstrass requiring locally compact
spaces.

Since $\mathcal{C}_{\Re }^{\Re ^{p}}\subset \mathcal{F}_{\Re }^{\Re ^{p}}$,
the proposed function space is much larger, and we cannot utilize such a
result here, which requires properties similar to Lemma \ref{AlgebraFuncs}.
Moreover, the space of functions $\mathcal{F}_{\Re }^{\Re ^{p}}$ satisfies
almost all requirements ($\Re ^{p}$ is Hausdorff), but $\Re ^{p}$ is not
compact. Instead, we will show directly the existence of a countable dense
subset\ of functions from $\mathcal{F}_{\Re }^{\Re ^{p}},$ utilizing
properties of measurable functions, so that a subspace of $\mathcal{F}_{\Re
}^{\Re ^{p}}$ becomes a Polish space.

Up to this point we have imposed and utilized the integrability conditions
of Remark \ref{CondRemarks} on the functions $a_{n}$ and $g,$ but as we see
next, we need additional specific conditions in order to prove that the
resulting restricted subspace of $\mathcal{F}_{\Re }^{\Re ^{p}}$\ is Polish.

\begin{theorem}[Polish Space]
\label{TaylorPolishMTMF}The space of multivariate Taylor measure functions $%
\mathcal{F}_{\Re }^{\Re ^{p}},$ \newline
equipped with the induced norm $\left\Vert .\right\Vert _{\rho _{\mathcal{F}%
}},$ restricted to $f^{B}\in T_{g(\mathbf{x}),\mathbf{a}(\mathbf{x})}(B)\in
\mathcal{F}_{\Re }^{\Re ^{p}},$ $B\in \mathcal{B}(\mathbb{N}),$ where the
integrals%
\begin{equation}
\int\limits_{\Re ^{p}}\left[ g^{n}(\mathbf{x})-a_{n}(\mathbf{x})+c\right]
^{2}\mu _{p}(d\mathbf{x})<+\infty ,  \label{IntegrabilityCondsPolish}
\end{equation}%
are defined and are finite, for all $n\in B,$ and any $c>0$, is a Polish
space.
\end{theorem}

\begin{proof}
Since $\mathcal{F}_{\Re }^{\Re ^{p}}$ is a complete metric space, it remains
to show that there exists a countable dense subset. Take any $B\in \mathcal{B%
}(\mathbb{N}),$ and consider arbitrary $f^{B}\in \mathcal{F}_{\Re }^{\Re
^{p}},$ where%
\begin{equation*}
f^{B}(\mathbf{x})=\sum\limits_{n\in B}a_{n}(\mathbf{x})\frac{g(\mathbf{x}%
)^{n}}{n!},
\end{equation*}%
with $a_{n},$ $n\in \mathbb{N},$ and $g$ measurable functions in $(\Re ^{p},%
\mathcal{B}(\Re ^{p}))$. If $g(\mathbf{x})=0,$ for some $\mathbf{x}\in \Re
^{p}$, we have $f^{B}(\mathbf{x})=0,$ and therefore, as long as the zero
function is a member of an entertained countable dense set the claim is
satisfied. The case where $0\in B$, is handled with the convention $0^{0}=1$%
. Consequently, we will assume without loss of generality that $g(\mathbf{x}%
)\neq 0,$ for all $\mathbf{x}\in \Re ^{p}$.\newline
Since $f^{B}\in \mathcal{F}_{\Re }^{\Re ^{p}},$ there exist monotone
sequences of measurable simple functions $\{a_{n,k}\}_{k=0}^{+\infty }$ and $%
\{g_{k}\}_{k=0}^{+\infty }$ such that $\underset{k\rightarrow +\infty }{\lim
}a_{n,k}(\mathbf{x})=a_{n}(\mathbf{x}),$ and $\underset{k\rightarrow +\infty
}{\lim }g_{k}(\mathbf{x})=g(\mathbf{x}),$ for each fixed $\mathbf{x}\in \Re
^{p}$, so that $\forall \varepsilon >0,\exists k_{1}\in \mathbb{N},$ such
that $\forall k>k_{1}$%
\begin{equation}
\left\vert a_{n,k}(\mathbf{x})-a_{n}(\mathbf{x})\right\vert <\varepsilon
\Leftrightarrow a_{n}(\mathbf{x})-\varepsilon <a_{n,k}(\mathbf{x})<a_{n}(%
\mathbf{x})+\varepsilon ,  \label{Anbounds}
\end{equation}%
and since $\underset{k\rightarrow +\infty }{\lim }g_{k}^{n}(\mathbf{x}%
)=g^{n}(\mathbf{x}),$ we have that $\forall \varepsilon >0,\exists k_{2}\in
\mathbb{N},$ such that $\forall k>k_{2}$%
\begin{equation}
\left\vert g_{k}^{n}(\mathbf{x})-g^{n}(\mathbf{x})\right\vert <\varepsilon
\Leftrightarrow g^{n}(\mathbf{x})-\varepsilon <g_{k}^{n}(\mathbf{x})<g^{n}(%
\mathbf{x})+\varepsilon ,  \label{Gammabounds}
\end{equation}%
for all $n\in \mathbb{N}.$ In particular, write%
\begin{equation*}
a_{n,k}(\mathbf{x})=\sum\limits_{l\in
L_{a}^{n,k}}c_{n,k}^{(l)}I_{C_{n,k}^{(l)}}(\mathbf{x}),
\end{equation*}%
for the partition $\{C_{n,k}^{(l)}\}_{l\in L_{a}^{n,k}}$ of $\Re ^{p},$ with
$C_{n,k}^{(l)}=\{\mathbf{x}\in \Re ^{p}:\mathbf{x}%
=a_{n,k}^{-1}(c_{n,k}^{(l)})\}\in \mathcal{B}(\Re ^{p}),$ $L_{a}^{n,k}$ a
countable set, and some real constants $c_{n,k}^{(l)},$ and%
\begin{equation*}
g_{k}(\mathbf{x})=\sum\limits_{l\in L_{\gamma
}^{k}}d_{k}^{(l)}I_{D_{k}^{(l)}}(\mathbf{x}),
\end{equation*}%
for the partition $\{D_{k}^{(l)}\}_{l\in L_{g}^{k}}$ of $\Re ^{p},$ with $%
D_{k}^{(l)}=\{\mathbf{x}\in \Re ^{p}:\mathbf{x}=g_{k}^{-1}(d_{k}^{(l)})\}\in
\mathcal{B}(\Re ^{p}),$ $L_{g}^{k}$ a countable set, and some real constants
$d_{k}^{(l)}.$ Since $\mathbb{Q}$ is dense in $\Re $, we can find sequences
of rationals $\{s_{j}^{(l,k)}\}_{j=1}^{+\infty }$ and $\{s_{j}^{(n,l,k)}%
\}_{j=1}^{+\infty },$ such that $\underset{j\rightarrow +\infty }{\lim }%
s_{j}^{(l,k)}=d_{k}^{(l)},$ for all $l\in L_{\gamma }^{k},$ and $\underset{%
j\rightarrow +\infty }{\lim }s_{j}^{(n,l,k)}=c_{n,k}^{(l)},$ for all $l\in
L_{a}^{n,k},$ and all $n,k\in \mathbb{N}.$ As a result, $\forall \varepsilon
>0,$ $\exists j_{1}\in \mathbb{N},$ such that $\forall j>j_{1}$%
\begin{equation}
\left\vert s_{j}^{(l,k)}-d_{k}^{(l)}\right\vert <\varepsilon \Leftrightarrow
d_{k}^{(l)}-\varepsilon <s_{j}^{(l,k)}<d_{k}^{(l)}+\varepsilon ,
\label{DenseEq1}
\end{equation}%
and similarly, $\forall \varepsilon >0,$ $\exists j_{2}\in \mathbb{N},$ such
that $\forall j>j_{2}$%
\begin{equation}
\left\vert s_{j}^{(n,l,k)}-c_{n,k}^{(l)}\right\vert <\varepsilon
\Leftrightarrow c_{n,k}^{(l)}-\varepsilon
<s_{j}^{(n,l,k)}<c_{n,k}^{(l)}+\varepsilon .  \label{DenseEq2}
\end{equation}%
\newline
Let $\mathcal{C}=\{T_{g(\mathbf{x}),\mathbf{a}(\mathbf{x})}\}\subset
\mathcal{F}_{\Re }^{\Re ^{p}},$ the collection of all measurable functions
from $\mathcal{F}_{\Re }^{\Re ^{p}},$ where the functions $g(\mathbf{x})$
and $a_{n}(\mathbf{x})$ are the simple functions with the rational
coefficients and partitions presented above, i.e., if $h_{j,k}^{B}\in
\mathcal{C}$, we have%
\begin{equation}
h_{j,k}^{B}(\mathbf{x})=\sum\limits_{n\in B}\left[ \sum\limits_{l\in
L_{a}^{n,k}}s_{j}^{(n,l,k)}I_{C_{n,k}^{(l)}}(\mathbf{x})\right] \left[
\sum\limits_{l\in L_{\gamma }^{k}}s_{j}^{(l,k)}I_{D_{k}^{(l)}}(\mathbf{x})%
\right] ^{n}\frac{1}{n!},  \label{HSimplefunction}
\end{equation}%
$j,k\in \mathbb{N}.$ Now note that, by construction, $\mathcal{C}$\ is a
countable collection of measurable functions in $\mathcal{F}_{\Re }^{\Re
^{p}}$, and define the collection of measurable functions which are limits
of sequences of elements from $\mathcal{C}$ by%
\begin{equation}
\mathcal{C}_{\lim }=\{f^{B}:\forall \varepsilon >0,\exists j_{0},k_{0}\in
\mathbb{N},\text{ such that }\forall j>j_{0},k>k_{0}:\left\Vert
f^{B}-h_{j,k}^{B}\right\Vert _{\rho _{\mathcal{F}}}<\varepsilon
,\{h_{j,k}^{B}\}\subset \mathcal{C}\}.  \label{ConvergenceDensePointwise}
\end{equation}%
\newline
Take $j>j_{0}=\max \{j_{1},j_{2}\},$ and $k>k_{0}=\max \{k_{1},k_{2}\},$ and
using equations (\ref{DenseEq1}) and (\ref{DenseEq2}), we write%
\begin{gather*}
(f^{B}-h_{j,k}^{B})(\mathbf{x})=\sum\limits_{n\in B}\left[ a_{n}(\mathbf{x}%
)g^{n}(\mathbf{x})-\sum\limits_{l\in
L_{a}^{n,k}}s_{j}^{(n,l,k)}I_{C_{n,k}^{(l)}}(\mathbf{x})\left[
\sum\limits_{l\in L_{\gamma }^{k}}s_{j}^{(l,k)}I_{D_{k}^{(l)}}(\mathbf{x})%
\right] ^{n}\right] \frac{1}{n!} \\
<\sum\limits_{n\in B}\left[ a_{n}(\mathbf{x})g^{n}(\mathbf{x}%
)-\sum\limits_{l\in L_{a}^{n,k}}\left( c_{n,k}^{(l)}-\varepsilon \right)
I_{C_{n,k}^{(l)}}(\mathbf{x})\left[ \sum\limits_{l\in L_{\gamma }^{k}}\left(
d_{k}^{(l)}+\varepsilon \right) I_{D_{k}^{(l)}}(\mathbf{x})\right] ^{n}%
\right] \frac{1}{n!} \\
<\sum\limits_{n\in B}\left[ a_{n}(\mathbf{x})g^{n}(\mathbf{x})-a_{n,k}(%
\mathbf{x})g_{k}^{n}(\mathbf{x})\right] \frac{1}{n!}=\sum\limits_{n\in B}%
\left[ a_{n}(\mathbf{x})-a_{n,k}(\mathbf{x})\frac{g_{k}^{n}(\mathbf{x})}{%
g^{n}(\mathbf{x})}\right] \frac{g^{n}(\mathbf{x})}{n!},
\end{gather*}%
so that%
\begin{eqnarray*}
\left\Vert f^{B}-h_{j,k}^{B}\right\Vert _{\rho _{\mathcal{F}}}^{2}
&<&\sum\limits_{n\in B}\int\limits_{\Re ^{p}}\left[ a_{n}(\mathbf{x}%
)-a_{n,k}(\mathbf{x})\frac{g_{k}^{n}(\mathbf{x})}{g^{n}(\mathbf{x})}\right]
^{2}\frac{g^{2n}(\mathbf{x})}{n!}\mu _{p}(d\mathbf{x}) \\
&=&\sum\limits_{n\in B}\int\limits_{\Re ^{p}}\left[ a_{n}(\mathbf{x})g^{n}(%
\mathbf{x})-a_{n,k}(\mathbf{x})g_{k}^{n}(\mathbf{x})\right] ^{2}\frac{1}{n!}%
\mu _{p}(d\mathbf{x}).
\end{eqnarray*}%
Now using equations (\ref{Anbounds}) and (\ref{Gammabounds}), we have%
\begin{gather*}
\left\Vert f^{B}-h_{j,k}^{B}\right\Vert _{\rho _{\mathcal{F}%
}}^{2}<\sum\limits_{n\in B}\int\limits_{\Re ^{p}}\left[ a_{n}(\mathbf{x}%
)g^{n}(\mathbf{x})-\left( a_{n}(\mathbf{x})-\varepsilon \right) \left( g^{n}(%
\mathbf{x})+\varepsilon \right) \right] ^{2}\frac{1}{n!}\mu _{p}(d\mathbf{x})
\\
=\sum\limits_{n\in B}\int\limits_{\Re ^{p}}\left[ a_{n}(\mathbf{x})g^{n}(%
\mathbf{x})-a_{n}(\mathbf{x})g^{n}(\mathbf{x})-\varepsilon a_{n}(\mathbf{x}%
)+\varepsilon g^{n}(\mathbf{x})+\varepsilon ^{2}\right] ^{2}\frac{1}{n!}\mu
_{p}(d\mathbf{x}) \\
=\varepsilon ^{2}\sum\limits_{n\in B}\frac{1}{n!}\int\limits_{\Re ^{p}}\left[
g^{n}(\mathbf{x})-a_{n}(\mathbf{x})+\varepsilon \right] ^{2}\mu _{p}(d%
\mathbf{x}),
\end{gather*}%
where the latter integrals are finite by assumption. Consequently, we have%
\begin{equation*}
\left\Vert f^{B}-h_{j,k}^{B}\right\Vert _{\rho _{\mathcal{F}%
}}^{2}<\varepsilon ^{2}\sum\limits_{n\in B}\frac{1}{n!}\int\limits_{\Re ^{p}}%
\left[ g^{n}(\mathbf{x})-a_{n}(\mathbf{x})+\varepsilon \right] ^{2}\mu _{p}(d%
\mathbf{x})
\end{equation*}%
\newline
and sending $\varepsilon \rightarrow 0,$ we have that $\left\Vert
f^{B}-h_{j,k}^{B}\right\Vert _{\rho _{\mathcal{M}}}\rightarrow 0,$ as $%
j,k\rightarrow +\infty ,$ so that $f^{B}\in \mathcal{C}_{\lim }$. As a
result, the closure of the countable set $\mathcal{C}$ is such that $%
\overline{\mathcal{C}}=\mathcal{C}\cup \mathcal{C}_{\lim }=\mathcal{F}_{\Re
}^{\Re ^{p}}$, and the restricted $\mathcal{F}_{\Re }^{\Re ^{p}}$ under
conditions (\ref{IntegrabilityCondsPolish}) is separable, as desired.
\end{proof}

The following lemma, provides some insight to pointwise convergence in $%
\mathcal{F}_{\Re }^{\Re ^{p}}$, under the standard norm in $\Re $. This
result is expected for a closed set of measurable functions, but we present
it here explicitly, since it can be used to define other important concepts,
such as integration with respect to MTMFs.

\begin{lemma}
\label{PointwiseLimit}For any arbitrary $f^{B}\in \mathcal{F}_{\Re }^{\Re
^{p}},$ there exists a sequence $\{h_{j,k}^{B}\}$ of elements of $\mathcal{F}%
_{\Re }^{\Re ^{p}}$, where $h_{j,k}^{B},$ $j,k\in \mathbb{N}$, is given by
equation (\ref{HSimplefunction}), such that $\underset{j,k\rightarrow
+\infty }{\lim }h_{j,k}^{B}(\mathbf{x})=f^{B}(\mathbf{x})$, for all $\mathbf{%
x}\in \Re ^{p}$.
\end{lemma}

\begin{proof}
Assume the notation in the proof of the previous theorem. We show that $%
h_{j,k}^{B}$ converges pointwise to $f^{B}$. We have%
\begin{gather*}
\underset{j,k\rightarrow +\infty }{\lim }h_{j,k}^{B}(\mathbf{x})=\underset{%
j,k\rightarrow +\infty }{\lim }\sum\limits_{n\in B}\left[ \sum\limits_{l\in
L_{a}^{n,k}}s_{j}^{(n,l,k)}I_{C_{n,k}^{(l)}}(\mathbf{x})\right] \left[
\sum\limits_{l\in L_{\gamma }^{k}}s_{j}^{(l,k)}I_{D_{k}^{(l)}}(\mathbf{x})%
\right] ^{n}\frac{1}{n!} \\
=\sum\limits_{n\in B}\underset{j,k\rightarrow +\infty }{\lim }\left[
\sum\limits_{l\in L_{a}^{n,k}}s_{j}^{(n,l,k)}I_{C_{n,k}^{(l)}}(\mathbf{x})%
\right] \left[ \underset{j,k\rightarrow +\infty }{\lim }\sum\limits_{l\in
L_{\gamma }^{k}}s_{j}^{(l,k)}I_{D_{k}^{(l)}}(\mathbf{x})\right] ^{n}\frac{1}{%
n!} \\
=\sum\limits_{n\in B}\underset{k\rightarrow +\infty }{\lim }\left[
\sum\limits_{l\in L_{a}^{n,k}}\underset{j\rightarrow +\infty }{\lim }%
s_{j}^{(n,l,k)}I_{C_{n,k}^{(l)}}(\mathbf{x})\right] \left[ \underset{%
k\rightarrow +\infty }{\lim }\sum\limits_{l\in L_{\gamma }^{k}}\underset{%
j\rightarrow +\infty }{\lim }s_{j}^{(l,k)}I_{D_{k}^{(l)}}(\mathbf{x})\right]
^{n}\frac{1}{n!}
\end{gather*}%
\begin{eqnarray*}
&=&\sum\limits_{n\in B}\underset{k\rightarrow +\infty }{\lim }\left[
\sum\limits_{l\in L_{a}^{n,k}}c_{n,k}^{(l)}I_{C_{n,k}^{(l)}}(\mathbf{x})%
\right] \left[ \underset{k\rightarrow +\infty }{\lim }\sum\limits_{l\in
L_{\gamma }^{k}}d_{k}^{(l)}I_{D_{k}^{(l)}}(\mathbf{x})\right] ^{n}\frac{1}{n!%
} \\
&=&\sum\limits_{n\in B}\underset{k\rightarrow +\infty }{\lim }a_{n,k}(%
\mathbf{x})\left[ \underset{k\rightarrow +\infty }{\lim }\gamma _{k}(\mathbf{%
x})\right] ^{n}\frac{1}{n!}=\sum\limits_{n\in B}a_{n}(\mathbf{x})\frac{%
\gamma ^{n}(\mathbf{x})}{n!}=f^{B}(\mathbf{x}),
\end{eqnarray*}%
for all fixed $\mathbf{x}\in \Re ^{p}$. Note here that convergence is
pointwise in $\mathbf{x}\in \Re ^{p}$, and that we invoke bounded
convergence theorem as required, in order to pass the limits under the
summation signs.
\end{proof}

We collect some special cases and properties of $\mathcal{F}_{\Re }^{\Re
^{p}}$ in the following.

\begin{remark}[Special Cases and Properties]
\label{MTMFExcont}We discuss some important consequences of the theoretical
development up to this point.

\begin{enumerate}
\item Continuity and subspaces of $\mathcal{F}_{\Re }^{\Re ^{p}}$: Let $B\in
\mathcal{B}(\mathbb{N}),$ and $\left\Vert .\right\Vert $ the Euclidean norm
in $\Re ^{p}.$ Recall that $f:(\Re ^{p},\left\Vert .\right\Vert )\rightarrow
(\Re ,|.|)$ is called Lipschitz if and only if there exists $K<\infty $ such
that%
\begin{equation*}
\left\vert f(\mathbf{x})-f(\mathbf{y})\right\vert \leq K\left\Vert \mathbf{x}%
-\mathbf{y}\right\Vert =K\sqrt{\sum\limits_{i=1}^{p}(x_{i}-y_{i})^{2}},
\end{equation*}%
for all $\mathbf{x},\mathbf{y}\in \Re ^{p}.$ Since sums and products of
Lipschitz functions are also Lipschitz, then any function $f^{B}\in \mathcal{%
F}_{\Re }^{\Re ^{p}},$ is also Lipschitzian, provided that $f^{B}(B)=T_{g(%
\mathbf{x}),\mathbf{a}(\mathbf{x})}(B),$ is defined based on Lipschitz
functions $a_{n},g:\Re ^{p}\rightarrow \Re ,$ where $\mathbf{a}(\mathbf{x}%
)=[a_{0}(\mathbf{x}),$ $a_{1}(\mathbf{x}),...],$ $n\in \mathbb{N}$. As a
result, the restriction of $\mathcal{F}_{\Re }^{\Re ^{p}}$ to members with
underlying Lipschitz functions $a_{n}$ and $g,$ describes the totality of
all Lipschitz functions. Now every Lipschitz continuous function is
uniformly continuous but the converse is not always true.\newline
Similarly as above, the totality of uniformly continuous functions is
obtained by the restriction of $\mathcal{F}_{\Re }^{\Re ^{p}}$ to members
with underlying uniformly continuous functions $a_{n}$ and $g.$ The same
arguments can be made for continuous functions and analytic functions.
Consequently, $\mathcal{F}_{\Re }^{\Re ^{p}}$ emerges as a unifying
framework in studying important collections of functions, since it contains
all of the aforementioned subspaces, with appropriate assumptions in the
underlying functions $a_{n}$ and $g$ used to define its members.

\item $\mathcal{L}^{p}$-spaces: In order to connect $\mathcal{F}_{\Re }^{\Re
^{r}}$ with $\mathcal{L}^{p}$-spaces, we require some additional assumptions
in order to handle unbounded functions and the integrals involved, e.g., for
the collection of all trivial MTMFs $f\in \mathcal{F}_{\Re }^{\Re ^{r}}$
(i.e., all measurable functions)$,$ we require that $\int\limits_{\Re
^{r}}|f|^{p}d\mu _{r}<+\infty ,$ $1\leq p<+\infty ,$ where $\mu _{r}$
denotes Lebesgue measure in $(\Re ^{r},\mathcal{B}(\Re ^{r}))$, and we do
not impose any other integrability conditions such as those of Remark \ref%
{CondRemarks}. We denote this restricted space by $\mathcal{F}_{r}=\mathcal{F%
}(\Re ^{r},\mathcal{B}(\Re ^{r}),\mu _{r})\subset \mathcal{F}_{\Re }^{\Re
^{r}}.$ Then using the $\mathcal{L}^{p}$-norm, $\left\Vert f\right\Vert
_{p}=\left( \int\limits_{\Re ^{r}}|f|^{p}d\mu _{r}\right) ^{\frac{1}{p}},$
we have immediately that $\mathcal{F}_{r}=\mathcal{L}^{p}(\Re ^{r},\mathcal{B%
}(\Re ^{r}),\mu _{r}),$ i.e., $\mathcal{F}_{r}$ is equivalent to Lebesgue $%
\mathcal{L}^{p}$-space, and thus enjoys all its properties.

\item Riemann's zeta function: Take $p=1$, $B=\mathbb{N}^{+},$ $a_{0}(x)=0,$
$a_{n}(x)=n!/n^{x},$ $n\in \mathbb{N}^{+},$ and $g(x)=1,$ so that equation (%
\ref{MultTaylorFunction}) yields Riemann's zeta function
\begin{equation}
\zeta (x)=\sum\limits_{n\in \mathbb{N}^{+}}\frac{1}{n^{x}}=\sum\limits_{n\in
B}a_{n}(x)\frac{g^{n}(x)}{n!}=T_{1,\mathbf{a}(x)}(\mathbb{N}^{+}),
\label{zetafunc}
\end{equation}%
and it converges for $x>1.$ A\ direct generalization to the multivariate
case is given by%
\begin{equation*}
\zeta (\mathbf{x})=\sum\limits_{n\in \mathbb{N}^{+}}\frac{1}{%
n^{x_{1}+...+x_{p}}}=T_{1,\mathbf{a}(\mathbf{x})}(\mathbb{N}^{+}),
\end{equation*}%
with $a_{0}(\mathbf{x})=0,$ $a_{n}(\mathbf{x})=n!/n^{x_{1}+...+x_{p}},$ $%
n\in \mathbb{N}^{+}$, which converges for $x_{1}+...+x_{p}>1.$. For a recent
treatment of the $\zeta $ function using series expansions see \cite%
{young2023global}.

\item Hypergeometric series: For $p=1$, set $B=\mathbb{N},$ $g(x)=x$, and $%
a_{n}(x)=n!c_{n},$ free of $x$, and further assume that%
\begin{equation*}
\frac{c_{n+1}}{c_{n}}=\frac{A(n)}{B(n)},
\end{equation*}%
with $A(n)$ and $B(n)$ polynomials in $n$, for all $n\in \mathbb{N}$. Then
the MTMF becomes the hypergeometric series%
\begin{equation*}
T_{x,\mathbf{a}(x)}(\mathbb{N})=\sum\limits_{n\in \mathbb{N}}c_{n}x^{n},
\end{equation*}%
which can be extended to the generalized hypergeometric function; recall the
Pochhammer symbols defined by%
\begin{eqnarray*}
(a)_{0} &=&1, \\
(a)_{n} &=&a(a+1)(a+2)...(a+n-1),\text{ }n\in \mathbb{N}^{+},
\end{eqnarray*}%
$a\in \Re ,$ and write%
\begin{equation*}
T_{x,\mathbf{c}}(\mathbb{N})=\text{ }%
_{p}F_{q}(a_{1},...,a_{p};b_{1},...,b_{q};x)=\sum\limits_{n\in \mathbb{N}%
}c_{n}\frac{x^{n}}{n!},
\end{equation*}%
with $\mathbf{c}=$ $[c_{0},c_{1},$ $\dots ],$ free of $x$, and%
\begin{equation*}
c_{n}=\frac{(a_{1})_{n}...(a_{p})_{n}}{(b_{1})_{n}...(b_{q})_{n}},
\end{equation*}%
for all $n\in \mathbb{N}.$ The MTMF $T_{x,\mathbf{c}(x)}(\mathbb{N})$
contains as special cases the Wilson polynomials (see \cite{wilson1980some}%
), and when the coefficients $c_{n}$ are allowed to be specific functions of
$x$, the Askey scheme (see \cite{brezinski2006polynomes}, \cite%
{askey1985some}, \cite{koekoek1994askey}) and the continuous Hahn
polynomials (see \cite{koekoek2010hypergeometric}), and their modifications
and extensions.

\item Generalization of Rodrigues' formula: Next take $B=\mathbb{N},$ and $%
a_{k}(\mathbf{x})=a_{n}h_{n}(\mathbf{x})h(\mathbf{x})e^{-g(\mathbf{x})},$
where $a_{n}\in \Re ,$ with $h$, $h_{n}$ and $g$ some analytic functions.
Then we can write%
\begin{equation*}
P_{n}(\mathbf{x})=T_{g(\mathbf{x}),\mathbf{a}(\mathbf{x})}(\mathbb{N}%
)=\sum\limits_{k\in \mathbb{N}}a_{k}(\mathbf{x})\frac{g(\mathbf{x})^{k}}{k!}%
=a_{n}h_{n}(\mathbf{x})h(\mathbf{x})e^{-g(\mathbf{x})}\sum\limits_{k\in
\mathbb{N}}\frac{g(\mathbf{x})^{k}}{k!},
\end{equation*}%
which yields%
\begin{equation}
P_{n}(\mathbf{x})=a_{n}h(\mathbf{x})h_{n}(\mathbf{x}),  \label{Iterative2}
\end{equation}%
for all $n\in \mathbb{N}.$ Note that $P_{n}$ as constructed is not
necessarily a polynomial. As a result, the latter becomes a generalization
of Rodrigues' formula
\begin{equation}
P_{n}(x)=\frac{c_{n}}{w(x)}\frac{d^{n}}{dx^{n}}\left[ B^{n}(x)w(x)\right] ,
\label{Rodriguesformula}
\end{equation}%
for $p=1$, $a_{n}=c_{n},$ $h(x)=1/w(x),$ and $h_{n}(x)=\frac{d^{n}}{dx^{n}}%
\left[ B^{n}(x)w(x)\right] ,$ under the required conditions on $c_{n},$ $w,$
and $B$ (see for example \cite{askey20051839})$.$ Consequently, we
immediately recognize a plethora of classic polynomials as special cases of
the MTMFs of equation (\ref{Iterative2}) (for $p=1$ and appropriate support,
not $\Re $ necessarily), and present them in Table \ref{PolysTable} (see
\cite{brezinski2006polynomes},\cite{castillo2023classical}, \cite%
{castillo2024first} for more details on these classic polynomials and more).
Note that the collection of polynomials $\{P_{n}(x)\}_{n=0}^{+\infty },$
where $P_{n}$ is a polynomial of degree $n$, form a sequence of orthogonal
polynomials on an interval $[a,b]$ with respect to the weighted inner product%
\begin{equation}
\int\limits_{a}^{b}P_{n}(x)P_{m}(x)w(x)\mu _{1}(dx)=K_{m,n}\delta _{mn},
\label{InnerProductPoly}
\end{equation}%
where $\delta _{mn}$ Kronecker's delta function, and $K_{m,n}$ are non zero
constants, for all $m,n\in \mathbb{N}$.

\item Monomials in $\mathcal{F}_{\Re }^{\Re ^{p}}$: A standard (countable)
basis in the space of polynomials includes the monomials $\{\mathbf{x}^{%
\mathbf{b}}\}_{\mathbf{b}\in \mathbb{N}^{p}},$ where $\mathbf{x}=$ $%
[x_{1},x_{2},$ $\dots ,x_{p}]\in \Re ^{p},$ $\mathbf{b}=$ $[b_{1},b_{2},$ $%
\dots ,b_{p}]\in \mathbb{N}^{p},$ and $\mathbf{x}^{\mathbf{b}%
}=x_{1}^{b_{1}}...x_{p}^{b_{p}},$ which are special cases of MTMFs since for
$B=\{1\}$, $g(\mathbf{x})=\mathbf{x}^{\mathbf{b}}$ and $\mathbf{a}(\mathbf{x}%
)=[0,1,0,...],$ we have%
\begin{equation}
\mathbf{x}^{\mathbf{b}}=\sum\limits_{n\in B}a_{n}(\mathbf{x})\frac{\left(
\mathbf{x}^{\mathbf{b}}\right) ^{n}}{n!}=T_{\mathbf{x}^{\mathbf{b}%
},[0,1,0,...]}(\{1\}).
\end{equation}%
Monomials are not orthogonal with respect to the weighted inner product (\ref%
{InnerProductPoly}), but can be made orthogonal using the Gram-Schmidt
orthogonalization process. For an analytic function $f\in \mathcal{F}_{\Re
}^{\Re ^{p}}$ at a point $\mathbf{x}_{0}=$ $[x_{1,0},x_{2,0},$ $\dots
,x_{p,0}]$ $\in $ $\Re ^{p}$, the monomials $\{(\mathbf{x}-\mathbf{x}_{0})^{%
\mathbf{b}}\}_{\mathbf{b}\in \mathbb{N}^{p}}$, can uniquely describe $f$ via
Taylor's theorem (special case of Definition \ref{MTMFDef}), but in other
cases, say continuous, non-analytic functions, one needs to use polynomials
(linear combinations of $\mathbf{x}^{\mathbf{b}}$), and apply Weierstrass'
theorem. Even though they can approximate continuous functions as close as
we want, polynomials fail to form a basis in $\mathcal{C}_{\Re }^{\Re ^{p}}.$
\end{enumerate}
\end{remark}

\begin{table}[tbp]
\centering{\tiny
\begin{tabular}{@{}lllll}
\hline
Polynomial & $a_n$ & $h(x)$ & $h_{n}(x)$ & Support \\ \hline
Chebyshev (1rst kind) & $\frac{(-1)^{n}2^{n}n!}{(2n)!}$ & $\sqrt{1-x^{2}}$ &
$\frac{d^{n}}{dx^{n}}\left[ (1-x^{2})^{n-\frac{1}{2}}\right]$ & $[-1,1]$ \\
Chebyshev (2nd kind) & $\frac{(-1)^{n}2^{n+1}(n+1)^{2}n!}{(2n+2)!}$ & $1/%
\sqrt{1-x^{2}}$ & $\frac{d^{n}}{dx^{n}}\left[ (1-x^{2})^{n+\frac{1}{2}}%
\right]$ & $[-1,1]$ \\
Hermite & $(-1)^{n}$ & $e^{x^{2}}$ & $\frac{d^{n}}{dx^{n}}e^{-x^{2}}$ & $\Re$
\\
Jacobi & $\frac{(-1)^{n}}{2^{n}n!}$ & $(1-x)^{-a}(1+x)^{-b}$ & $\frac{d^{n}}{%
dx^{n}}\left( (1-x)^{a}(1+x)^{b}\left( 1-x^{2}\right) ^{n}\right)$ & $[-1,1]$
\\
Laguerre & $\frac{1}{n!}$ & $e^x$ & $\frac{d^{n}}{dx^{n}}\left(
x^{n}e^{-x}\right)$ & $[0,+\infty)$ \\
Legendre & $\frac{1}{2^{n}n!}$ & $1$ & $\frac{d^{n}}{dx^{n}}\left(
x^{2}-1\right)^n$ & $[-1,1]$ \\ \hline
\end{tabular}
}
\caption{Classic orthogonal polynomials as special cases of MTMFs based on
equation (\protect\ref{Iterative2}).}
\label{PolysTable}
\end{table}

Next we discuss differentiation of functions in $\mathcal{F}_{\Re }^{\Re
^{p}}.$

\subsection{Differentiation in $\mathcal{F}_{\Re }^{\Re ^{p}}$}

Recall that a measurable function is not necessarily differentiable, and
therefore, investigating differentiation in\ the function space $\mathcal{F}%
_{\Re }^{\Re ^{p}}$ will require additional assumptions. In what follows,
let $c_{k}^{n}=\frac{n!}{(n-k)!k!},$ and $c_{k_{1},...,k_{n}}^{n}=\frac{n!}{%
k_{1}!...k_{n}!},$ the binomial and multinomial coefficients, respectively$.$
For arbitrary $B\in \mathcal{B}(\mathbb{N})$, and $f^{B}\in \mathcal{F}_{\Re
}^{\Re ^{p}}$, given by%
\begin{equation}
f^{B}(\mathbf{x})=T_{g(\mathbf{x}),\mathbf{a}(\mathbf{x})}(B)=\sum\limits_{n%
\in B}a_{n}(\mathbf{x})\frac{g^{n}(\mathbf{x})}{n!},  \label{DiffTMFgen}
\end{equation}%
assume that $\mathbf{a}=$ $[a_{0}(\mathbf{x}),a_{1}(\mathbf{x}),$ $\dots ],$
has analytic elements and let $g$ be analytic, so that $f^{B}(\mathbf{x})$
is continuously, differentiable of any order and with respect to any of its
arguments. In particular, differentiation with respect to $x_{j}$ of
equation (\ref{DiffTMFgen}), and a straightforward application of the
general Leibnitz rule yields%
\begin{equation*}
\frac{\partial ^{k}}{\partial x_{j}^{k}}T_{g(\mathbf{x}),\mathbf{a}(\mathbf{x%
})}(B)=\sum\limits_{n\in B}\frac{1}{n!}\frac{\partial ^{k}}{\partial
x_{j}^{k}}\left[ a_{n}(\mathbf{x})g^{n}(\mathbf{x})\right]
=\sum\limits_{n\in B}\frac{1}{n!}\sum\limits_{l=0}^{k}c_{l}^{k}\frac{%
\partial ^{k-l}a_{n}(\mathbf{x})}{\partial ^{k-l}x_{j}}\frac{\partial ^{l}}{%
\partial x_{j}^{l}}\left[ g^{n}(\mathbf{x})\right] ,
\end{equation*}%
for all $k\in \mathbb{N}$. Since%
\begin{equation*}
\frac{\partial ^{l}g(\mathbf{x})^{n}}{\partial x_{j}^{l}}=\sum\limits
_{\substack{ l_{1}+...+l_{n}=l  \\ l_{1},...,l_{n}\geq 0}}%
c_{l_{1},...,l_{n}}^{n}\prod\limits_{1\leq d\leq n}\frac{\partial ^{k_{d}}g(%
\mathbf{x})}{\partial ^{k_{d}}x_{j}},
\end{equation*}%
we have the general form for differentiation of elements of $\mathcal{F}%
_{\Re }^{\Re ^{p}},$ given by%
\begin{equation}
\frac{\partial ^{k}}{\partial x_{j}^{k}}T_{g(\mathbf{x}),\mathbf{a}(\mathbf{x%
})}(B)=\sum\limits_{n\in B}\frac{1}{n!}\sum\limits_{l=0}^{k}c_{l}^{k}\frac{%
\partial ^{k-l}a_{n}(\mathbf{x})}{\partial ^{k-l}x_{j}}\sum\limits
_{\substack{ l_{1}+...+l_{n}=l  \\ l_{1},...,l_{n}\geq 0}}%
c_{l_{1},...,l_{n}}^{n}\prod\limits_{1\leq d\leq n}\frac{\partial ^{k_{d}}g(%
\mathbf{x})}{\partial ^{k_{d}}x_{j}}.  \label{DerivFormula1}
\end{equation}

Next consider a special case of MTMFs, where for any $f^{B}\in \mathcal{F}%
_{\Re }^{\Re ^{p+1}},$ we set%
\begin{equation}
f^{B}(t,\mathbf{x})=T_{g(\mathbf{x}),\mathbf{a}_{t}}(B)=\sum\limits_{n\in
B}a_{n}(t)\frac{g(\mathbf{x})^{n}}{n!},  \label{DiffTMF}
\end{equation}%
where $\mathbf{x}=[x_{1},...,x_{p}]\in \Re ^{p},$ and $\mathbf{a}%
_{t}=[a_{0}(t),a_{1}(t),...]\in \Re ^{\infty },$ has elements that are
functions of only $t\in \Re $, $g$ is a function of only $\mathbf{x}\in \Re
^{p}$, and $B\in \mathcal{B}(\mathbb{N})$. Let the shift operator of the
sequence $\mathbf{a}_{t}$ be defined by $\tau _{B}(\mathbf{a}%
_{t})=[a_{l+1}(t),l\in B],$ and define $\tau _{B}^{n}=\tau _{B}\circ
...\circ \tau _{B},$ the $n$-fold composition of $\tau _{B}$ with itself. In
particular, $\tau _{\mathbb{N}}^{n}(\mathbf{a}%
_{t})=[a_{n+1}(t),a_{n+2}(t),...],$ $n\in \mathbb{N}.$

We use $\mathcal{A}_{\Re }^{\Re ^{p+1}}$ to denote the subspace of functions
$\mathcal{F}_{\Re }^{\Re ^{p+1}}$ of the form (\ref{DiffTMF}), with $g$ an
analytic function, so that all order derivatives exist, with respect to all
its arguments. In some cases, we can relax this strong assumption on $g$,
and obtain important formulas for MTMFs, as we see next.

\begin{remark}[Finite order differentiation]
\label{MTMFremark1}Assume that $B_{N}=\{0,1,...,N\},$ with $N\in \mathbb{N},$
in general. However, noting that $T_{g(\mathbf{x}),\mathbf{a}%
_{t}}(B_{0})=a_{0}(t),$ and $T_{g(\mathbf{x}),\tau _{B_{0}}(\mathbf{a}%
_{t})}(B_{0})$ $=$ $a_{1}(t),$ free of $\mathbf{x}\in \Re ^{p}$, we will
assume in certain places that $N>1$.

\begin{enumerate}
\item Second order: When $g$ is continuously, twice differentiable, with
respect to all its arguments, taking derivative with respect to $x_{i}$ of
equation (\ref{DiffTMF}) we have%
\begin{eqnarray*}
\frac{\partial f^{B_{N}}(t,\mathbf{x})}{\partial x_{i}} &=&\frac{\partial
T_{g(\mathbf{x}),\mathbf{a}_{t}}(B_{N})}{\partial x_{i}}=\frac{\partial g(%
\mathbf{x})}{\partial x_{i}}\sum\limits_{n=0}^{N}a_{n}(t)n\frac{g(\mathbf{x}%
)^{n-1}}{n!} \\
&=&\frac{\partial g(\mathbf{x})}{\partial x_{i}}\sum\limits_{n=1}^{N}a_{n}(t)%
\frac{g(\mathbf{x})^{n-1}}{(n-1)!}\overset{k=n-1}{=}\frac{\partial g(\mathbf{%
x})}{\partial x_{i}}\sum\limits_{k=0}^{N-1}a_{k+1}(t)\frac{g(\mathbf{x})^{k}%
}{k!}
\end{eqnarray*}%
so that we have an iterative formula for the first partial derivative given
by%
\begin{equation}
\frac{\partial f^{B_{N}}(t,\mathbf{x})}{\partial x_{i}}=\frac{\partial g(%
\mathbf{x})}{\partial x_{i}}T_{g(\mathbf{x}),\tau _{B_{N-1}}(\mathbf{a}%
_{t})}(B_{N-1})\in \mathcal{F}_{\Re }^{\Re ^{p+1}}.
\label{FirstPartialDeriv}
\end{equation}%
Differentiating with respect to $x_{j}$ yields%
\begin{gather}
\frac{\partial ^{2}f^{B_{N}}(t,\mathbf{x})}{\partial x_{j}\partial x_{i}}=%
\frac{\partial }{\partial x_{j}}\left( \frac{\partial g(\mathbf{x})}{%
\partial x_{i}}T_{g(\mathbf{x}),\tau _{B_{N-1}}(\mathbf{a}%
_{t})}(B_{N-1})\right)  \notag \\
=\frac{\partial ^{2}g(\mathbf{x})}{\partial x_{j}\partial x_{i}}T_{g(\mathbf{%
x}),\tau _{B_{N-1}}(\mathbf{a}_{t})}(B_{N-1})+\frac{\partial g(\mathbf{x})}{%
\partial x_{i}}\frac{\partial g(\mathbf{x})}{\partial x_{j}}T_{g(\mathbf{x}%
),\tau _{B_{N-2}}^{2}(\mathbf{a}_{t})}(B_{N-2})\in \mathcal{F}_{\Re }^{\Re
^{p+1}}.  \label{SecondPartialDeriv}
\end{gather}

\item Finite order: Assume that $g$ is continuously, $k$th order
differentiable, with respect to all its arguments. From equation (\ref%
{FirstPartialDeriv}), an appeal to the general Leibniz rule yields%
\begin{eqnarray*}
\frac{\partial ^{k+1}T_{g(\mathbf{x}),\mathbf{a}_{t}}(B_{N})}{\partial
x_{i}^{k+1}} &=&\frac{\partial ^{k}}{\partial x_{i}^{k}}\left( \frac{%
\partial g(\mathbf{x})}{\partial x_{i}}T_{g(\mathbf{x}),\tau _{B_{N-1}}(%
\mathbf{a}_{t})}(B_{N-1})\right) \\
&=&\sum\limits_{n=0}^{k}c_{n}^{k}\frac{\partial ^{k-n+1}g(\mathbf{x})}{%
\partial ^{k-n+1}x_{i}}\frac{\partial ^{n}}{\partial x_{i}^{n}}T_{g(\mathbf{x%
}),\tau _{B_{N-1}}(\mathbf{a}_{t})}(B_{N-1}),
\end{eqnarray*}%
with $N>1$, which provides an iterative formula for calculation of
derivatives of MTMFs of all order with respect to $x_{i}$.\newline
On the other hand, using equation (\ref{DiffTMF}), we can write
\begin{equation}
\frac{\partial ^{k+1}T_{g(\mathbf{x}),\mathbf{a}_{t}}(B_{N})}{\partial
x_{i}^{k+1}}=\sum\limits_{n=0}^{N}\frac{a_{n}(t)}{n!}\frac{\partial ^{k+1}g(%
\mathbf{x})^{n}}{\partial x_{i}^{k+1}},  \label{DerivMTMF1}
\end{equation}%
and since%
\begin{equation*}
\frac{\partial ^{k}g(\mathbf{x})^{n}}{\partial x_{i}^{k}}=\sum\limits
_{\substack{ k_{1}+...+k_{n}=k  \\ k_{1},...,k_{n}\geq 0}}%
c_{k_{1},...,k_{n}}^{k}\prod\limits_{l=1}^{n}\frac{\partial ^{k_{l}}g(%
\mathbf{x})}{\partial ^{k_{l}}x_{i}},
\end{equation*}%
we have an explicit formula of the $k$th order derivative of $T_{g(\mathbf{x}%
),\mathbf{a}_{t}}(B_{N})$ in terms of $g$ and $\mathbf{a}_{t}$ given by%
\begin{equation}
\frac{\partial ^{k+1}T_{g(\mathbf{x}),\mathbf{a}_{t}}(B_{N})}{\partial
x_{i}^{k+1}}=\sum\limits_{n=0}^{N}\frac{a_{n}(t)}{n!}\sum\limits_{\substack{ %
k_{1}+...+k_{n}=k+1  \\ k_{1},...,k_{n}\geq 0}}c_{k_{1},...,k_{n}}^{k+1}%
\prod\limits_{l=1}^{n}\frac{\partial ^{k_{l}}g(\mathbf{x})}{\partial
^{k_{l}}x_{i}},
\end{equation}%
for all $k\in \mathbb{N}$. Therefore, for $N>1$, we must have%
\begin{equation}
\sum\limits_{n=0}^{k}c_{n}^{k}\frac{\partial ^{k-n+1}g(\mathbf{x})}{\partial
^{k-n+1}x_{i}}\frac{\partial ^{n}y_{t}(\mathbf{x})}{\partial x_{i}^{n}}%
=\sum\limits_{n=0}^{N}\frac{a_{n}(t)}{n!}\sum\limits_{\substack{ %
k_{1}+...+k_{n}=k+1  \\ k_{1},...,k_{n}\geq 0}}c_{k_{1},...,k_{n}}^{k+1}%
\prod\limits_{l=1}^{n}\frac{\partial ^{k_{l}}g(\mathbf{x})}{\partial
^{k_{l}}x_{i}},  \label{DifferentiallFormulaT}
\end{equation}%
so that the latter is a linear, inhomogeneous, $k$th order ordinary
differential equation (ODE) with respect to $x_{i}$, with a partial solution%
\begin{equation*}
y_{t}(\mathbf{x})=T_{g(\mathbf{x}),\tau _{B_{N-1}}(\mathbf{a}%
_{t})}(B_{N-1})=\sum\limits_{k=0}^{N-1}a_{k+1}(t)\frac{g(\mathbf{x})^{k}}{k!}%
,
\end{equation*}%
$i=1,2,...,p$. Sending $N\rightarrow \infty ,$ and assuming all function
sums on the right hand side below converge, we obtain the space-time ODE%
\begin{equation}
\sum\limits_{n=0}^{k}c_{n}^{k}\frac{\partial ^{k-n+1}g(\mathbf{x})}{\partial
^{k-n+1}x_{i}}\frac{\partial ^{n}y_{t}(\mathbf{x})}{\partial x_{i}^{n}}%
=\sum\limits_{n=0}^{+\infty }\frac{a_{n}(t)}{n!}\sum\limits_{\substack{ %
k_{1}+...+k_{n}=k+1  \\ k_{1},...,k_{n}\geq 0}}c_{k_{1},...,k_{n}}^{k+1}%
\prod\limits_{l=1}^{n}\frac{\partial ^{k_{l}}g(\mathbf{x})}{\partial
^{k_{l}}x_{i}},  \label{MTMFODE1}
\end{equation}%
with a partial solution the MTMF given by%
\begin{equation}
y_{t}(\mathbf{x})=\sum\limits_{n=0}^{+\infty }a_{n+1}(t)\frac{g(\mathbf{x}%
)^{n}}{n!}.  \label{MTMFODESol1}
\end{equation}
\end{enumerate}
\end{remark}

\begin{example}[Space-Time ODEs]
We consider specific cases of the linear ODE (\ref{MTMFODE1}) and the MTMF
of equation (\ref{MTMFODESol1}), in order to illustrate first applications
to partial differential equations. In what follows we take $p=1$, $k=2$, and
set $B=\mathbb{N}$, i.e., we have the ODE in the form%
\begin{equation*}
g^{\prime }(x)\frac{\partial ^{2}y_{t}(x)}{\partial x^{2}}+2g^{\prime \prime
}(x)\frac{\partial y_{t}(x)}{\partial x}+g^{\prime \prime \prime
}(x)y_{t}(x)=\frac{\partial ^{3}q_{t}(x)}{\partial x^{3}},
\end{equation*}%
using Equation (\ref{DerivMTMF1}), with%
\begin{equation*}
q_{t}(x)=T_{g(x),\mathbf{a}_{t}}(\mathbb{N})=\sum\limits_{n=0}^{+\infty
}a_{n}(t)\frac{g(x)^{n}}{n!},
\end{equation*}%
and%
\begin{equation*}
y_{t}(x)=\sum\limits_{n=0}^{+\infty }a_{n+1}(t)\frac{g(x)^{n}}{n!}.
\end{equation*}%
We look for an MTMF that satisfies%
\begin{equation}
\frac{\partial ^{3}q_{t}(x)}{\partial x^{3}}=\lambda \frac{\partial y_{t}(x)%
}{\partial t}=\lambda \sum\limits_{n=0}^{+\infty }\frac{\partial a_{n+1}(t)}{%
\partial t}\frac{g(x)^{n}}{n!},  \label{MTMFDerivODE1}
\end{equation}%
for some constant $\lambda >0,$ i.e., we conjecture that $y_{t}(x)$ is the
partial solution to the partial differential equation (PDE) given by%
\begin{equation}
g^{\prime }(x)\frac{\partial ^{2}y_{t}(x)}{\partial x^{2}}+2g^{\prime \prime
}(x)\frac{\partial y_{t}(x)}{\partial x}+g^{\prime \prime \prime
}(x)y_{t}(x)=\lambda \frac{\partial y_{t}(x)}{\partial t},
\label{2ndspacetimeODE}
\end{equation}%
for some\ $g(x)$ and $a_{n}(t)$, provided we can find such functions. Using
analogous arguments as above, we may consider PDEs of the form%
\begin{equation}
g^{\prime }(x)\frac{\partial ^{2}y_{t}(x)}{\partial x^{2}}+2g^{\prime \prime
}(x)\frac{\partial y_{t}(x)}{\partial x}+g^{\prime \prime \prime
}(x)y_{t}(x)=\xi ^{2}\frac{\partial ^{2}y_{t}(x)}{\partial t^{2}},
\label{2ndspacetimewaveODE}
\end{equation}%
where $\xi \neq 0.$

\begin{enumerate}
\item Heat Equation: Take $g(x)=x$, so that this special case of the PDE\ (%
\ref{2ndspacetimeODE}) becomes the celebrated heat equation%
\begin{equation}
\frac{\partial ^{2}y_{t}(x)}{\partial x^{2}}=\lambda \frac{\partial y_{t}(x)%
}{\partial t},  \label{HeatEq1}
\end{equation}%
$\forall x\in \lbrack 0,a],$ and $a,$ $t>0,$ with partial solution%
\begin{equation*}
y_{t}(x)=\sum\limits_{n=0}^{+\infty }a_{n+1}(t)\frac{x^{n}}{n!},
\end{equation*}%
which is not separable in general (i.e., recall that separability implies
that for some measurable functions $w_{1}(t)$ and $w_{2}(x)$ we have $%
y_{t}(x)=w_{1}(t)w_{2}(x),$ see \cite{chasnov2009introduction}, Section
8.5), and%
\begin{equation*}
q_{t}(x)=\sum\limits_{n=0}^{+\infty }a_{n}(t)\frac{x^{n}}{n!}.
\end{equation*}%
Then assuming that $\mathbf{a}(t)$ are such that passing derivatives with
respect to $x$ or $t$, under the summation sign is valid, we write%
\begin{eqnarray*}
\frac{\partial ^{2}y_{t}(x)}{\partial x^{2}} &=&\frac{\partial ^{3}q_{t}(x)}{%
\partial x^{3}}=\frac{\partial ^{2}}{\partial x^{2}}\sum\limits_{n=0}^{+%
\infty }a_{n}(t)n\frac{x^{n-1}}{n!}=\frac{\partial }{\partial x}%
\sum\limits_{n=0}^{+\infty }a_{n}(t)n(n-1)\frac{x^{n-2}}{n!} \\
&=&\sum\limits_{n=0}^{+\infty }a_{n}(t)n(n-1)(n-2)\frac{x^{n-3}}{n!}%
=\sum\limits_{n=3}^{+\infty }a_{n}(t)\frac{x^{n-3}}{(n-3)!}\overset{k=n-3}{=}%
\sum\limits_{k=0}^{+\infty }a_{k+3}(t)\frac{x^{k}}{k!},
\end{eqnarray*}%
and using Equation (\ref{MTMFDerivODE1}), we have that (\ref{HeatEq1}) yields%
\begin{equation*}
\sum\limits_{n=0}^{+\infty }a_{n+3}(t)\frac{x^{n}}{n!}=\lambda
\sum\limits_{n=0}^{+\infty }\frac{\partial a_{n+1}(t)}{\partial t}\frac{x^{n}%
}{n!},\text{ }\forall x\in \lbrack 0,a].
\end{equation*}%
Consequently, we must have%
\begin{eqnarray*}
a_{3}(t) &=&\lambda \frac{\partial a_{1}(t)}{\partial t},\text{ }n=0, \\
a_{4}(t) &=&\lambda \frac{\partial a_{2}(t)}{\partial t},\text{ }n=1, \\
a_{5}(t) &=&\lambda \frac{\partial a_{3}(t)}{\partial t}=\lambda ^{2}\frac{%
\partial ^{2}a_{1}(t)}{\partial t^{2}},\text{ }n=2, \\
a_{6}(t) &=&\lambda \frac{\partial a_{4}(t)}{\partial t}=\lambda ^{2}\frac{%
\partial ^{2}a_{2}(t)}{\partial t^{2}},\text{ }n=3,
\end{eqnarray*}%
and so forth, that is, the functions $a_{n}(t),$ $n\in \mathbb{N},$ must
satisfy%
\begin{eqnarray*}
a_{2n+1}(t) &=&\lambda ^{n}\frac{\partial ^{n}a_{1}(t)}{\partial t^{n}},%
\text{ }n=1,2,...,\text{ and} \\
a_{2n+2}(t) &=&\lambda ^{n}\frac{\partial ^{n}a_{2}(t)}{\partial t^{n}},%
\text{ }n=1,2,...,
\end{eqnarray*}%
$t>0,$ for unknown measurable functions $a_{0}(t),$ $a_{1}(t)$ and $%
a_{2}(t), $ that can obtained once we impose boundary conditions with
respect to space (e.g., $a_{0}(t)$ is arbitrary, and for $x=0,$ $%
y_{t}(0)=a_{1}(t),$ and $\frac{\partial y_{t}(0)}{\partial t}=a_{2}(t)$).
Consequently, the partial solution to the heat equation is given by%
\begin{equation*}
y_{t}(x)=\sum\limits_{n=0}^{+\infty }a_{n+1}(t)\frac{x^{n}}{n!}%
=\sum\limits_{n=0}^{+\infty }a_{2n+1}(t)\frac{x^{2n}}{(2n)!}%
+\sum\limits_{n=0}^{+\infty }a_{2n+2}(t)\frac{x^{2n+1}}{(2n+1)!},
\end{equation*}%
or%
\begin{equation*}
y_{t}(x)=\sum\limits_{n=0}^{+\infty }\frac{\partial ^{n}a_{1}(t)}{\partial
t^{n}}\frac{\left( \sqrt{\lambda }x\right) ^{2n}}{(2n)!}+\sum%
\limits_{n=0}^{+\infty }\frac{\partial ^{n}a_{2}(t)}{\partial t^{n}}\frac{%
\left( \sqrt{\lambda }x\right) ^{2n+1}}{(2n+1)!},
\end{equation*}%
$\forall x\in \lbrack 0,a],$ and $a,$ $t>0.$ Clearly, the partial solution
to the heat equation in this case requires analytic functions $a_{1}(t),$
and $a_{2}(t)$, as well as conditions such as those of Definition \ref%
{MTMFDef}\ on the sequences $\frac{n!\lambda ^{n}}{(2n)!}\frac{\partial
^{n}a_{1}(t)}{\partial t^{n}},$ and $\frac{n!\lambda ^{n}}{(2n+1)!}\frac{%
\partial ^{n}a_{2}(t)}{\partial t^{n}}$, in order for $y_{t}(x)$ to be well
defined. In particular, for $a_{1}(t)=e^{-t},$ and $a_{2}(t)=0,$ $\forall
t>0,$ we have the MTMF%
\begin{equation*}
y_{t}(x)=e^{-t}\sum\limits_{n=0}^{+\infty }(-1)^{n}\frac{(\sqrt{\lambda }%
x)^{2n}}{(2n)!}=e^{-t}\cos \left( \sqrt{\lambda }x\right) ,
\end{equation*}%
with $\frac{\partial y_{t}(x)}{\partial x}=-\sqrt{\lambda }e^{-t}\sin \left(
\sqrt{\lambda }x\right) ,$ $\frac{\partial ^{2}y_{t}(x)}{\partial x^{2}}%
=-\lambda e^{-t}\cos \left( \sqrt{\lambda }x\right) ,$ and $\frac{\partial
y_{t}(x)}{\partial t}=-y_{t}(x),$ so that (\ref{HeatEq1}) is satisfied.
Similarly, for $a_{1}(t)=0,$ and $a_{2}(t)=e^{-t},$ $\forall t>0,$ we have
the MTMF%
\begin{equation*}
y_{t}(x)=e^{-t}\sum\limits_{n=0}^{+\infty }(-1)^{n}\frac{(\sqrt{\lambda }%
x)^{2n+1}}{(2n+1)!}=e^{-t}\sin \left( \sqrt{\lambda }x\right) ,
\end{equation*}%
with $\frac{\partial y_{t}(x)}{\partial x}=\sqrt{\lambda }e^{-t}\cos (\sqrt{%
\lambda }x),$ $\frac{\partial ^{2}y_{t}(x)}{\partial x^{2}}=-\lambda
e^{-t}\sin (\sqrt{\lambda }x),$ and $\frac{\partial y_{t}(x)}{\partial t}%
=-y_{t}(x),$ so that (\ref{HeatEq1}) is once again satisfied.

\item Wave Equation: Take again $g(x)=x$, so that this special case of the
PDE\ (\ref{2ndspacetimewaveODE}) becomes the famous wave equation%
\begin{equation}
\frac{\partial ^{2}y_{t}(x)}{\partial x^{2}}=\xi ^{2}\frac{\partial
^{2}y_{t}(x)}{\partial t^{2}},  \label{WaveEquation}
\end{equation}%
$\forall x\in \lbrack 0,a],$ and $a,$ $t>0,$ and proceeding as in the
previous example, we have that (\ref{WaveEquation}) yields%
\begin{equation*}
\sum\limits_{n=0}^{+\infty }a_{n+3}(t)\frac{x^{n}}{n!}=\xi
^{2}\sum\limits_{n=0}^{+\infty }\frac{\partial ^{2}a_{n+1}(t)}{\partial t^{2}%
}\frac{x^{n}}{n!},\text{ }\forall x\in \lbrack 0,a].
\end{equation*}%
Thus, we must have%
\begin{eqnarray*}
a_{3}(t) &=&\xi ^{2}\frac{\partial ^{2}a_{1}(t)}{\partial t^{2}},\text{ }n=0,
\\
a_{4}(t) &=&\xi ^{2}\frac{\partial ^{2}a_{2}(t)}{\partial t^{2}},\text{ }n=1,
\\
a_{5}(t) &=&\xi ^{2}\frac{\partial ^{2}a_{3}(t)}{\partial t^{2}}=\xi ^{4}%
\frac{\partial ^{4}a_{1}(t)}{\partial t^{4}},\text{ }n=2, \\
a_{6}(t) &=&\xi ^{2}\frac{\partial ^{2}a_{4}(t)}{\partial t^{2}}=\xi ^{4}%
\frac{\partial ^{4}a_{2}(t)}{\partial t^{4}},\text{ }n=3,
\end{eqnarray*}%
and so forth, that is, the functions $a_{n}(t),$ $n\in \mathbb{N},$ must
satisfy%
\begin{eqnarray*}
a_{2n+1}(t) &=&\xi ^{2n}\frac{\partial ^{2n}a_{1}(t)}{\partial t^{2n}},\text{
}n=1,2,...,\text{ and} \\
a_{2n+2}(t) &=&\xi ^{2n}\frac{\partial ^{2n}a_{2}(t)}{\partial t^{2n}},\text{
}n=1,2,...,
\end{eqnarray*}%
$t>0.$ Then, the partial solution to the wave equation is given by%
\begin{eqnarray*}
y_{t}(x) &=&\sum\limits_{n=0}^{+\infty }a_{n+1}(t)\frac{x^{n}}{n!}%
=\sum\limits_{n=0}^{+\infty }a_{2n+1}(t)\frac{x^{2n}}{(2n)!}%
+\sum\limits_{n=0}^{+\infty }a_{2n+2}(t)\frac{x^{2n+1}}{(2n+1)!} \\
&=&\sum\limits_{n=0}^{+\infty }\xi ^{2n}\frac{\partial ^{2n}a_{1}(t)}{%
\partial t^{2n}}\frac{x^{2n}}{(2n)!}+\sum\limits_{n=0}^{+\infty }\xi ^{2n}%
\frac{\partial ^{2n}a_{2}(t)}{\partial t^{2n}}\frac{x^{2n+1}}{(2n+1)!},
\end{eqnarray*}%
$\forall x\in \lbrack 0,a],$ and $a,$ $t>0.$ The functions $a_{1}(t)$ and $%
a_{2}(t)$, must satisfy similar conditions as in the previous example, in
order for $y_{t}(x)$ to be well defined.\newline
Once again we obtain separability by setting $a_{1}(t)=e^{-t},$ and $%
a_{2}(t)=0,$ $\forall t>0,$ which leads to the MTMF solution given by%
\begin{equation*}
y_{t}(x)=e^{-t}\sum\limits_{n=0}^{+\infty }(-1)^{n}\frac{\left( \xi x\right)
^{2n}}{(2n)!}=e^{-t}\cos (\xi x),
\end{equation*}%
so that (\ref{WaveEquation}) is satisfied. Similarly, for $\lambda =1,$ $%
a_{1}(t)=0,$ and $a_{2}(t)=e^{-t},$ $\forall t>0,$ we have the MTMF solution
given by%
\begin{equation*}
y_{t}(x)=e^{-t}\sum\limits_{n=0}^{+\infty }(-1)^{n}\frac{(\xi x)^{2n+1}}{%
(2n+1)!}=e^{-t}\sin (\xi x),
\end{equation*}%
such that (\ref{WaveEquation}) is also satisfied.
\end{enumerate}
\end{example}

Next, we propose a general framework that provides a connection of
multivariate Taylor measure functions and differential equations.

\section{Taylor Measures and Differential Equations}

Our goal in this section is to find general conditions on the structure of
MTMFs, in order for them to be solutions to general differential equations
of specific forms. Before we begin, we consider a modification of Definition
(\ref{MTMFDef}) to functions on the complex plane $\mathbb{C}$. We use $%
\mathbb{M}$ to denote $\Re $ or $\mathbb{C}$, and consider a function $f^{B}:%
\mathbb{M}^{p}\rightarrow \mathbb{M}$, defined via the finite, signed Taylor
measure $f^{B}(\mathbf{x})=T_{g(\mathbf{x}),\mathbf{a}(\mathbf{x}%
)}(B)<+\infty ,$ where $\mathbf{a}(\mathbf{x})=[a_{0}(\mathbf{x}),$ $a_{1}(%
\mathbf{x}),...]\in \mathbb{M}^{\infty },$ has elements that are real or
complex functions$,$ and $a_{n},g:\mathbb{M}^{p}\rightarrow \mathbb{M},$ are
real or complex, measurable functions in $(\mathbb{M}^{p},\mathcal{B}(%
\mathbb{M}^{p}))$, for all $B\in \mathcal{B}(\mathbb{N}),$ and $n\in \mathbb{%
N}.$ The function $f^{B}$ will still be referred to as a multivariate Taylor
measure function, and we use the notation $\mathcal{F}_{\mathbb{M}}^{\mathbb{%
M}^{p}}$ to denote the corresponding space of functions. Finally, when $p=1$%
, we write $f^{\prime }(x)$ to denote $\frac{df(x)}{dx}.$

\subsection{Linear ordinary differential equations and MTMFs}

Assume that $p=1$, and consider the general framework for linear ordinary
differential equations (LODE). Let $L_{m}[\cdot ]$ the usual $m$th order
linear operator defined by%
\begin{equation}
L_{m}[y]=p_{m}(x)\frac{d^{m}y}{dx^{n}}+p_{m-1}(x)\frac{d^{m-1}y}{dx^{m-1}}%
+...+p_{1}(x)\frac{dy}{dx}+p_{0}(x)y,  \label{LinOperatornthorder}
\end{equation}%
where the $\{p_{i}(x)\}_{i=0}^{m}$ are $\mathbb{M}-$valued, continuous
functions and $p_{m}(x)\neq 0,$ on the interval $x\in I=[a,b]\subset \mathbb{%
M}$. Define $m$ initial conditions%
\begin{equation}
C_{j}^{(1)}[y]:=\frac{d^{j}y}{dx^{j}}(x_{0})=A_{j},\text{ }j=0,1,2,...,m-1
\label{Initial1}
\end{equation}%
for some constants $A_{j}\in \mathbb{M},$ and $x_{0}\in (a,b),$ and write $%
\mathcal{C}^{(1)}$ as shorthand for all conditions in (\ref{Initial1}), and $%
m$ (linear) boundary conditions by%
\begin{equation}
C_{j}^{(2)}[y]:=\sum\limits_{r=0}^{m-1}\left( M_{jr}\frac{d^{r}y}{dx^{r}}%
(a)+N_{jr}\frac{d^{r}y}{dx^{r}}(b)\right) =0,\text{ }j=1,2,...,m,
\label{Boundary1}
\end{equation}%
where $\{M_{jk},N_{jk}\}$ are given complex constants, and write $\mathcal{C}%
^{(2)}$ as shorthand for all conditions in (\ref{Boundary1}).

We will consider the linear ordinary differential equation of the form%
\begin{equation}
L_{m}[y]=h(x),  \label{LODE1}
\end{equation}%
where $h:\mathbb{M}\rightarrow \mathbb{M}$, a known continuous function in $(%
\mathbb{M},\mathcal{B}(\mathbb{M})),$ with $h(x)=T_{h_{g}(x),\mathbf{h}%
(x)}(B),$ for a given $B\in \mathcal{B}(\mathbb{N}),$ $\mathbf{h}%
(x)=[h_{0}(x),$ $h_{1}(x),$ $...]$ $\in $ $\mathbb{M}^{\infty },$ so that%
\begin{equation*}
h(x)=\sum\limits_{n\in B}h_{n}(x)\frac{h_{g}(x)^{n}}{n!},
\end{equation*}%
where we will assume that $h,h_{n},h_{g}\neq 0,$ are continuous functions.
In addition, we will impose either conditions $C^{(1)}$ or $C^{(2)}$ and
study the LODE\ problems%
\begin{equation}
L_{m}[y]=h(x),\text{ under conditions }\mathcal{C}^{(1)},
\label{InitialLODE1}
\end{equation}%
or the boundary value problem (BVP) given by%
\begin{equation}
L_{m}[y]=h(x),\text{ under conditions }\mathcal{C}^{(2)}.  \label{BVPLODE1}
\end{equation}

Note here that since $\{p_{i}(x)\}_{i=0}^{m}$ and $h(x)$ are all assumed to
be continuous, then there exists a unique solution to the problem (\ref%
{InitialLODE1}) (e.g., see \cite{kelley2010theory}, Theorem 6.2), however,
for the BVP (\ref{BVPLODE1}) additional conditions are required depending on
the forms of the conditions (see \cite{zwillinger2021handbook}, pg. 268).

In order to connect MTMFs with LODEs, we assume that $y=f^{B}(x)=T_{g(x),%
\mathbf{a}(x)}(B)\in \mathcal{F}_{\mathbb{M}}^{\mathbb{M}},$ is a solution
of the LODE (\ref{LODE1}), for the given $B\in \mathcal{B}(\mathbb{N})$,
that is used in the definition of $h(x).$ By the linearity property of $L$
we have%
\begin{equation*}
L_{m}[T_{g(x),\mathbf{a}(x)}(B)]=L_{m}\left[ \sum\limits_{n\in B}a_{n}(x)%
\frac{g(x)^{n}}{n!}\right] =\sum\limits_{n\in B}\frac{1}{n!}L_{m}\left[
a_{n}g^{n}\right] =\sum\limits_{n\in B}\frac{1}{n!}h_{n}(x)h_{g}(x)^{n},
\end{equation*}%
and using equation (\ref{DerivFormula1}) we have a first formula for the
operator $L_{m}$ in terms of the unknown measurable functions $g(x)$ and $%
\mathbf{a}(x),$ given by%
\begin{equation}
L_{m}\left[ a_{n}g^{n}\right] =\sum\limits_{k=0}^{m}p_{k}(x)\sum%
\limits_{l=0}^{k}c_{l}^{k}\frac{\partial ^{k-l}a_{n}(x)}{\partial ^{k-l}x}%
\sum\limits_{l_{1}+...+l_{n}=l}c_{l_{1},...,l_{n}}^{l}\prod\limits_{d=1}^{n}%
\frac{\partial ^{k_{d}}g(x)}{\partial ^{k_{d}}x}.  \label{LODE2}
\end{equation}%
We look for solutions to the problem (\ref{LODE1}), where $g(x)$ and $%
\mathbf{a}(x)$ are unknown functions that satisfy%
\begin{equation}
L_{m}\left[ a_{n}g^{n}\right] =h_{n}(x)h_{g}(x)^{n},  \label{MTMFODEs}
\end{equation}%
for given continuous functions $h_{n}$ and $h_{g},$ and all $n\in B.$ The
differential equation (\ref{MTMFODEs}) will be called the MTMF ODE of the $m$%
th order, for each $n\in B$.

The following result encapsulates the steps in obtaining the solution of the
problem (\ref{LODE1}) in terms of a MTMF $T_{g(x),\mathbf{a}(x)}(B)$, by
combining the solutions to the differential equation (\ref{MTMFODEs}), as
well as, illustrates how the solution is connected to the underlying
functions $g(x)$ and $\mathbf{a}(x)$. In particular, the solution to the
problem (\ref{LODE1}) is a MTMF for a specific $B\in \mathcal{B}(\mathbb{N}%
), $ and functions $a_{n}(x)$ and $g(x)$, that need not be analytic, but
they need to satisfy weaker conditions, e.g., be continuously,
differentiable up to the $m$th order. We will write $W(\mathbf{z}_{m})$ as
short hand for the Wronskian $W(z_{1},...,z_{m}),$ and similarly, $W_{k}(%
\mathbf{z}_{m})$ for $W_{k}(z_{1},...,z_{m}),$ the Wronskian $W(\mathbf{z}%
_{m})$ with its $k$-column replaced by $(0,0,...,0,1)$, where $\mathbf{z}%
_{m}=\{z_{1},...,z_{m}\},$ a collection of linearly independent functions ($%
W(\mathbf{z}_{m})\neq 0 $)$.$

\begin{theorem}[MTMF solution to ODEs]
\label{SolODEsThm}Consider the LODE (\ref{LODE1}), where $y=f^{B}(x)=T_{g(x),%
\mathbf{a}(x)}(B)=\sum\limits_{n\in B}a_{n}(x)\frac{g(x)^{n}}{n!}\in
\mathcal{F}_{\mathbb{M}}^{\mathbb{M}},$ for some $B\in \mathcal{B}(\mathbb{N}%
),$ and measurable functions $a_{n}(x)$ and $g(x)$, that are continuously,
differentiable up to the $m$th order. Then, assuming that the integrability
conditions (\ref{Integrable1}) hold, the general solution to the LODE (\ref%
{LODE1}) exists and is given by (\ref{GenSol2}). The functions $g(x)$ and $%
a_{n}(x)$, are given by Equations (\ref{Gfunc1}) and (\ref{a_nrelation2sols}%
), respectively.\newline
For the problem \ref{InitialLODE1} there exists a unique solution which is
of the form of the MTMF of Equation (\ref{GenSol2}).
\end{theorem}

\begin{proof}
\textit{Individual LODEs for each }$n\in B$\textit{:} Let $%
y_{n}(x)=a_{n}(x)g(x)^{n},$ so that the differential equation (\ref{MTMFODEs}%
) can be written as%
\begin{equation}
L_{m}\left[ y_{n}\right] =p_{m}(x)\frac{d^{m}y_{n}}{dx^{n}}+p_{m-1}(x)\frac{%
d^{m-1}y_{n}}{dx^{m-1}}+...+p_{1}(x)\frac{dy_{n}}{dx}%
+p_{0}(x)y_{n}=h_{n}(x)h_{g}(x)^{n},  \label{LODEindividual}
\end{equation}%
or equivalently, we need to solve the differential equation%
\begin{equation}
\frac{d^{m}y_{n}}{dx^{n}}+\frac{p_{m-1}(x)}{p_{m}(x)}\frac{d^{m-1}y_{n}}{%
dx^{m-1}}+...+\frac{p_{1}(x)}{p_{m}(x)}\frac{dy_{n}}{dx}+\frac{p_{0}(x)}{%
p_{m}(x)}y_{n}=\frac{h_{n}(x)}{p_{m}(x)}h_{g}(x)^{n}.
\label{DiffEqWronkians1}
\end{equation}%
\newline
\textit{Existence and Uniqueness:} Since the functions $\{p_{i}(x)%
\}_{i=0}^{m},$ $h_{n}(x)$ and $h_{g}(x)$, are continuous, there exists a
solution to the LODE (\ref{LODE1}), and it is unique under conditions $%
\mathcal{C}^{(1)}.$ Under conditions $\mathcal{C}^{(2)},$ uniqueness depends
on the specific forms of the BVP (e.g., for special cases see \cite%
{zwillinger2021handbook}, pg. 268).\newline
\textit{Variation of Parameters:} There are many different approaches to
solving (\ref{DiffEqWronkians1}), but we will adopt the variation of
parameters method (more details on this classic method can be found in \cite%
{zwillinger2021handbook}). Consider the homogeneous version of (\ref%
{DiffEqWronkians1}), given by%
\begin{equation}
\frac{d^{m}y_{n}}{dx^{n}}+\frac{p_{m-1}(x)}{p_{m}(x)}\frac{d^{m-1}y_{n}}{%
dx^{m-1}}+...+\frac{p_{1}(x)}{p_{m}(x)}\frac{dy_{n}}{dx}+\frac{p_{0}(x)}{%
p_{m}(x)}y_{n}=0,  \label{DiffEqWronkiansHomo1}
\end{equation}%
and assume that $\{z_{1},...,z_{m}\}$ is a set of linearly independent
solutions of (\ref{DiffEqWronkiansHomo1})$.$ Note that the latter is the
same set regardless of the value of $n,$ and each $z_{i}\in \mathcal{F}_{%
\mathbb{M}}^{\mathbb{M}},$ $i=1,2,...,m$, using the trivial MTMF
representation of equation (\ref{MTFsinglefunction}).\newline
Then all solutions of (\ref{DiffEqWronkians1}) are given by%
\begin{equation}
z_{n}(x)=z_{0,n}(x)+\sum\limits_{i=1}^{m}c_{i}z_{i}(x)\in \mathcal{F}_{%
\mathbb{M}}^{\mathbb{M}},  \label{GenSol1}
\end{equation}%
$x\in I$, where $z_{0,n}(x)$ is a partial solution of (\ref{DiffEqWronkians1}%
), and $c_{1},c_{2},...,c_{m}$ are constants that will be determined from
the conditions $\mathcal{C}^{(1)}$ or $\mathcal{C}^{(2)}$. The partial
solution can be written in terms of Wronskians via
\begin{equation}
z_{0,n}(x)=\sum\limits_{k=1}^{m}z_{k}(x)\int\limits_{x_{0}}^{x}\frac{W_{k}(%
\mathbf{z}_{m})(t)}{W(\mathbf{z}_{m})(t)}\frac{h_{n}(t)}{p_{m}(t)}%
h_{g}(t)^{n}\mu _{1}(dt)\in \mathcal{F}_{\mathbb{M}}^{\mathbb{M}},
\label{PartialSol1}
\end{equation}%
for any $x_{0}\in I$ and all $n\in B$. By construction, $z_{0,n}(x)$
satisfies the initial conditions%
\begin{equation}
\frac{d^{m-1}z_{0,n}(x_{0})}{dx^{m-1}}=...=z_{0,n}(x_{0})=0.
\label{Conditionsz0}
\end{equation}%
\newline
\textit{Integrability Requirements:} In view of (\ref{PartialSol1}), we will
require the conditions
\begin{equation}
u_{k,m}(x)=\frac{h(x)W_{k}(\mathbf{z}_{m})(x)}{p_{m}(x)W(\mathbf{z}_{m})(x)},%
\text{ are integrable with respect to }\mu _{1},  \label{Integrable1}
\end{equation}%
for all $k=1,2,...,m,$ in order for the solutions to be well defined (note
that (\ref{Integrable1}) implies that $\frac{W_{k}(\mathbf{z}_{m})(t)}{W(%
\mathbf{z}_{m})(t)}\frac{h_{n}(t)}{p_{m}(t)}h_{g}(t)^{n}$ is integrable for
all $n\in B).$\newline
\textit{Aggregating the individual Solutions:} Now combining the solutions (%
\ref{GenSol1}), for each $n\in B$ (see \cite{BirkhoffRota1991}, Lemma 2, pg
41)$,$ we obtain the general solution of the form%
\begin{equation}
T_{g(x),\mathbf{a}(x)}(B)=\sum\limits_{n\in B}\frac{1}{n!}%
z_{n}(x)=\sum\limits_{n\in B}\frac{1}{n!}\left[ z_{0,n}(x)+\sum%
\limits_{i=1}^{m}c_{i}z_{i}(x)\right] ,  \label{GenSol2}
\end{equation}%
where the constants $c_{1},c_{2},...,c_{m}$ are specified by the conditions $%
\mathcal{C}^{(1)}$ or $\mathcal{C}^{(2)}$.

\textit{Connection to the MTMF:} In order to appreciate the relationship
between the unknown functions $a_{n}$ and $g,$ and the MTMF of Equation (\ref%
{GenSol2}), we exemplify their connection next. Since $%
y_{n}(x)=a_{n}(x)g(x)^{n},$ is the solution to the differential equation (%
\ref{DiffEqWronkians1}), using Equation (\ref{GenSol1}), we have%
\begin{equation}
a_{n}(x)g(x)^{n}=z_{n}(x)=z_{0,n}(x)+\sum\limits_{i=1}^{m}c_{i}z_{i}(x),
\label{Connectaz}
\end{equation}%
for all $x\in \mathbb{M}.$ For $x\neq x_{0},$ we can write%
\begin{equation*}
y_{n}(x)=\sum\limits_{k=1}^{m}z_{k}(x)\int\limits_{x_{0}}^{x}\frac{W_{k}(%
\mathbf{z}_{m})(t)}{p_{m}(t)W(\mathbf{z}_{m})(t)}h_{n}(t)h_{g}(t)^{n}\mu
_{1}(dt)+\sum\limits_{k=1}^{m}c_{k}\frac{z_{k}(x)}{x-x_{0}}%
\int\limits_{x_{0}}^{x}\mu _{1}(dt),
\end{equation*}%
so that the solution $y_{n}$ can be written in terms of Green's function $%
w_{n,m}(x,t,\mathbf{z}_{m})$ (see \cite{zwillinger2021handbook}, pg 318) as%
\begin{equation*}
y_{n}(x)=\int\limits_{x_{0}}^{x}\sum\limits_{k=1}^{m}z_{k}(x)\left[ \frac{%
W_{k}(\mathbf{z}_{m})(t)}{p_{m}(t)W(\mathbf{z}_{m})(t)}+\frac{c_{k}}{\left(
x-x_{0}\right) h_{n}(t)h_{g}(t)^{n}}\right] h_{n}(t)h_{g}(t)^{n}\mu _{1}(dt),
\end{equation*}%
and thus allowing us to utilize the corresponding theory.\newline
In order to obtain tractable results for any $x\in \mathbb{M},$ letting%
\begin{equation}
w_{n,m}(x,x_{0},\mathbf{z}_{m})=\sum\limits_{k=1}^{m}z_{k}(x)\left[
\int\limits_{x_{0}}^{x}\frac{W_{k}(\mathbf{z}_{m})(t)}{p_{m}(t)W(\mathbf{z}%
_{m})(t)}h_{n}(t)h_{g}(t)^{n}\mu _{1}(dt)+c_{k}\right] ,  \label{wnm}
\end{equation}%
with $w_{n,m}(x_{0},x_{0},\mathbf{z}_{m})=\sum%
\limits_{k=1}^{m}c_{k}z_{k}(x_{0}),$ and using equation (\ref{PartialSol1}),
we write equation (\ref{Connectaz}) as%
\begin{eqnarray*}
y_{n}(x) &=&a_{n}(x)g(x)^{n}=z_{0,n}(x)+\sum\limits_{i=1}^{m}c_{i}z_{i}(x) \\
&=&\sum\limits_{k=1}^{m}z_{k}(x)\int\limits_{x_{0}}^{x}\frac{W_{k}(\mathbf{z}%
_{m})(t)}{p_{m}(t)W(\mathbf{z}_{m})(t)}h_{n}(t)h_{g}(t)^{n}\mu
_{1}(dt)+\sum\limits_{k=1}^{m}c_{k}z_{k}(x),
\end{eqnarray*}%
so that the solution $y_{n}$ can be written as%
\begin{equation}
y_{n}(x)=a_{n}(x)g(x)^{n}=w_{n,m}(x,x_{0},\mathbf{z}_{m}),  \label{SolGreen1}
\end{equation}%
where $a_{n}(x)$ and $g(x)$ are unknown measurable functions, not
necessarily uniquely determined since there are many combinations of
functions that can satisfy (\ref{SolGreen1}), and therefore, we have some
flexibility in the way we choose to represent the LODE solution as an MTMF.
This is expected since the MTMF representation is not unique, but as the
functions $y_{n}(x)$ are aggregated in (\ref{GenSol2}) to give the LODE
solution, they lead to a uniquely determined function. More precisely, the
collection of functions $a_{n}(x)$ and $g(x)$ satisfying (\ref{SolGreen1})
form an equivalence class%
\begin{equation*}
\lbrack a_{n},g]=\left\{ u_{n},u\in \mathcal{F}_{\mathbb{M}}^{\mathbb{M}%
}:u_{n}(x)u(x)^{n}=a_{n}(x)g(x)^{n}\right\} ,
\end{equation*}%
where $a_{n}$ and $g$ satisfy (\ref{SolGreen1}), i.e., $[a_{n},g]\symbol{126}%
[u_{n},u]$ is a reflexive, symmetric and transitive equivalence relation.%
\newline
We illustrate below how to obtain one such pair of functions $a_{n}\ $and $g$%
. When $g$ is chosen, we have immediately that%
\begin{equation}
a_{n}(x)=\frac{w_{n,m}(x,x_{0},\mathbf{z}_{m})}{g(x)^{n}},
\end{equation}%
for all $n\in B$. Let%
\begin{equation}
u_{m,k,n}(x)=\frac{h_{n}(x)W_{k}(\mathbf{z}_{m})(x)}{p_{m}(x)W(\mathbf{z}%
_{m})(x)},  \label{Ukmn}
\end{equation}%
and%
\begin{equation}
q_{n,m}(x,x_{0},\mathbf{z}_{m})=\frac{1}{h_{g}(x)^{n}}\frac{\partial }{%
\partial x}w_{n,m}(x,x_{0},\mathbf{z}_{m}).  \label{Qnm}
\end{equation}%
Differentiating (\ref{SolGreen1}) with respect to $x$ yields
\begin{gather*}
a_{n}^{\prime }(x)g(x)^{n}+na_{n}(x)g^{\prime }(x)g(x)^{n-1}=\left(
a_{n}^{\prime }(x)+n\frac{g^{\prime }(x)}{g(x)}a_{n}(x)\right) g(x)^{n} \\
=\frac{\partial }{\partial x}w_{n,m}(x,x_{0},\mathbf{z}_{m})=\sum%
\limits_{k=1}^{m}\frac{dz_{k}(x)}{dx}\left[ \int\limits_{x_{0}}^{x}\frac{%
W_{k}(\mathbf{z}_{m})(t)}{p_{m}(t)W(\mathbf{z}_{m})(t)}h_{n}(t)h_{g}(t)^{n}%
\mu _{1}(dt)+c_{k}\right]  \\
+\sum\limits_{k=1}^{m}z_{k}(x)\frac{W_{k}(\mathbf{z}_{m})(x)}{p_{m}(x)W(%
\mathbf{z}_{m})(x)}h_{n}(x)h_{g}(x)^{n} \\
=\sum\limits_{k=1}^{m}\frac{dz_{k}(x)}{dx}\left[ \int%
\limits_{x_{0}}^{x}u_{m,k,n}(t)h_{g}(t)^{n}\mu _{1}(dt)+c_{k}\right]
+\sum\limits_{k=1}^{m}z_{k}(x)u_{m,k,n}(x)h_{g}(x)^{n},
\end{gather*}%
and collecting $h_{g}(x)^{n}$ from both summands on the right hand side, we
have%
\begin{gather*}
\left( a_{n}^{\prime }(x)+n\frac{g^{\prime }(x)}{g(x)}a_{n}(x)\right)
g(x)^{n}=h_{g}(x)^{n} \\
\left[ \sum\limits_{k=1}^{m}\frac{1}{h_{g}(x)^{n}}\frac{dz_{k}(x)}{dx}\left[
\int\limits_{x_{0}}^{x}u_{m,k,n}(t)h_{g}(t)^{n}\mu _{1}(dt)+c_{k}\right]
+\sum\limits_{k=1}^{m}z_{k}(x)u_{m,k,n}(x)\right]  \\
=h_{g}(x)^{n}q_{n,m}(x,x_{0},\mathbf{z}_{m}).
\end{gather*}%
In view of the latter equation and (\ref{LODEindividual}), an intuitively
appealing choice for $g$ is to select
\begin{equation}
g(x)=h_{g}(x),  \label{Gfunc1}
\end{equation}%
and as a result, $a_{n}$ must satisfy%
\begin{equation}
a_{n}(x)=\frac{w_{n,m}(x,x_{0},\mathbf{z}_{m})}{h_{g}(x)^{n}}.
\label{a_nrelation2sols}
\end{equation}%
Finally, note that by construction, $a_{n}$ is the solution to the first
order ODE%
\begin{equation*}
a_{n}^{\prime }(x)+n\frac{h_{g}^{\prime }(x)}{h_{g}(x)}%
a_{n}(x)=q_{n,m}(x,x_{0},\mathbf{z}_{m}),
\end{equation*}%
$x_{0},x\in I.$
\end{proof}

In order to illustrate the importance of the MTMF (\ref{GenSol2}) solution
to the ODE (\ref{LODE1}), we consider some special cases in the following
example.

\begin{example}[Special Cases]
We consider the LODE (\ref{LODE1}) and the MTMFs of equations (\ref%
{DiffTMFgen}) and (\ref{DiffTMF}), under different scenarios.

\begin{enumerate}
\item Classic ODE form: Take $B=\{1\}$, and $g(x)=1$, so that $y(x)=T_{1,%
\mathbf{a}(x)}(\{1\})=a(x),$ where $\mathbf{a}(x)=[0,$ $a(x),$ $0,...].$
Then the problem (\ref{LODE1}) reduces to the usual linear ODE given by
\begin{equation*}
L_{m}[a]=h(x),
\end{equation*}%
where $h(x)=h_{1}(x)h_{g}(x),$ with solution given in Theorem \ref%
{SolODEsThm}.

\item MTMF decomposition illustrated: Consider an inhomogeneous Euler LODE
of the form%
\begin{equation}
y^{\prime \prime }(x)-2y^{\prime }(x)+y=\frac{1}{x}e^{x},  \label{Eulerex1}
\end{equation}%
$x>0,$ with the corresponding HODE having linearly independent solutions $%
y_{1}(x)=e^{x},$ and $y_{2}(x)=xe^{x},$ $x>0$. A partial solution is%
\begin{equation*}
y_{3}(x)=xe^{x}(\log x-1),
\end{equation*}%
$x>0$, so that the general solution is given by
\begin{equation}
y(x)=c_{1}e^{x}+c_{2}xe^{x}+xe^{x}(\log x-1)=e^{x}\left[ c_{1}+c_{2}x+x(\log
x-1)\right]   \label{ExEulergensol1}
\end{equation}%
for $x>0,$ and arbitrary constants $c_{1},c_{2}\in \Re $. We will show how
to obtain the latter via the MTMF of Theorem \ref{SolODEsThm}.\newline
Start by writing the operator $L_{2}[\cdot ]$ of Equation (\ref%
{LinOperatornthorder}) as%
\begin{equation*}
L_{2}[y]=p_{2}(x)y^{\prime \prime }+p_{1}(x)y^{\prime }+p_{0}(x)y,
\end{equation*}%
with $p_{2}(x)=1,$ $p_{1}(x)=-2,$ and $p_{0}(x)=1,$ and Equation (\ref{LODE1}%
) as%
\begin{equation*}
L_{2}[y]=\frac{1}{x}e^{x}=h(x),
\end{equation*}%
so that the right hand side of (\ref{Eulerex1}) is the function%
\begin{equation*}
h(x)=\frac{1}{x}e^{x}=\sum\limits_{n=0}^{+\infty }\frac{1}{x}\frac{x^{n}}{n!}%
=\sum\limits_{n\in B}h_{n}(x)\frac{h_{g}(x)^{n}}{n!},
\end{equation*}%
and consequently, $B=\mathbb{N},$ with $h_{n}(x)=1/x,$ and $h_{g}(x)=x,$ $%
x>0,$ continuous functions$.$ As the general solution to (\ref{Eulerex1}),
we entertain the MTMF given by%
\begin{equation*}
T_{g(x),\mathbf{a}(x)}(B)=\sum\limits_{n\in B}a_{n}(x)\frac{g(x)^{n}}{n!},
\end{equation*}%
with $B=\mathbb{N},$ $g(x)=h_{g}(x)=x,$ and $a_{n}(x)$ satisfies (\ref%
{a_nrelation2sols}). Now the individual ODEs of Equation (\ref{MTMFODEs})
take the form%
\begin{equation*}
L_{2}\left[ a_{n}g^{n}\right] =h_{n}(x)h_{g}(x)^{n}=\frac{1}{x}x^{n}=x^{n-1},
\end{equation*}%
so that the differential equation (\ref{DiffEqWronkians1}) is given by%
\begin{equation*}
y^{\prime \prime }(x)-2y^{\prime }(x)+y=x^{n-1}.
\end{equation*}%
Let $\{z_{1},z_{2}\}$ a set of linearly independent solutions of the
homogeneous ODE of the latter. We take $z_{1}(x)=e^{x},$ and $%
z_{2}(x)=xe^{x},$ $x>0$, with Wronskian%
\begin{equation*}
W(z_{1},z_{2})(x)=\left\vert
\begin{tabular}{ll}
$z_{1}(x)$ & $z_{2}(x)$ \\
$\frac{dz_{1}(x)}{dx}$ & $\frac{dz_{2}(x)}{dx}$%
\end{tabular}%
\right\vert =e^{2x}\neq 0,
\end{equation*}%
for all $x>0$, $W_{1}(z_{1},z_{2})(x)=xe^{x},$ and $%
W_{2}(z_{1},z_{2})(x)=e^{x}$. Then we write%
\begin{equation*}
\int\limits_{x_{0}}^{x}\frac{W_{1}(\mathbf{z}_{2})(t)}{p_{2}(t)W(\mathbf{z}%
_{2})(t)}h_{n}(t)h_{g}(t)^{n}\mu _{1}(dt)=\int\limits_{x_{0}}^{x}\frac{te^{t}%
}{e^{2t}}t^{n-1}\mu _{1}(dt)=\int\limits_{x_{0}}^{x}t^{n}e^{-t}\mu _{1}(dt),
\end{equation*}%
and%
\begin{equation*}
\int\limits_{x_{0}}^{x}\frac{W_{2}(\mathbf{z}_{2})(t)}{p_{2}(t)W(\mathbf{z}%
_{2})(t)}h_{n}(t)h_{g}(t)^{n}\mu _{1}(dt)=\int\limits_{x_{0}}^{x}\frac{e^{t}%
}{e^{2t}}t^{n-1}\mu _{1}(dt)=\int\limits_{x_{0}}^{x}t^{n-1}e^{-t}\mu
_{1}(dt),
\end{equation*}%
so that the partial solution is given by%
\begin{equation*}
z_{0,n}(x)=xe^{x}\int\limits_{x_{0}}^{x}t^{n}e^{-t}\mu
_{1}(dt)+e^{x}\int\limits_{x_{0}}^{x}t^{n-1}e^{-t}\mu _{1}(dt),
\end{equation*}%
and the general solution is given by%
\begin{equation*}
a_{n}(x)g(x)^{n}=z_{0,n}(x)+\sum\limits_{i=1}^{m}c_{i}z_{i}(x),
\end{equation*}%
where the constants $c_{1}$ and $c_{2},$ are specified by the conditions $%
\mathcal{C}^{(1)}$ or $\mathcal{C}^{(2)}$. Now after some rudimentary
calculations, Equation (\ref{wnm}) becomes
\begin{equation*}
w_{n,2}(x,x_{0},\mathbf{z}_{2})=e^{x}\left[ x\int%
\limits_{x_{0}}^{x}t^{n}e^{-t}\mu
_{1}(dt)+\int\limits_{x_{0}}^{x}t^{n-1}e^{-t}\mu _{1}(dt)\right]
+e^{x}(c_{1}+c_{2}x).
\end{equation*}%
We verify that the MTMF $T_{x,\mathbf{a}(x)}(\mathbb{N}),$ gives the general
solution; write%
\begin{eqnarray*}
T_{x,\mathbf{a}(x)}(\mathbb{N}) &=&\sum\limits_{n=0}^{+\infty }a_{n}(x)\frac{%
g(x)^{n}}{n!}=\sum\limits_{n=0}^{+\infty }\frac{1}{n!}w_{n,2}(x,x_{0},%
\mathbf{z}_{2}) \\
&=&e^{x}\left[ x\int\limits_{x_{0}}^{x}\sum\limits_{n=0}^{+\infty }\frac{%
t^{n}}{n!}e^{-t}\mu
_{1}(dt)+\int\limits_{x_{0}}^{x}\sum\limits_{n=0}^{+\infty }\frac{t^{n}}{n!}%
t^{-1}e^{-t}\mu _{1}(dt)\right] +e^{x}(c_{1}+c_{2}x)\sum\limits_{n=0}^{+%
\infty }\frac{1}{n!} \\
&=&e^{x}\left[ x\int\limits_{x_{0}}^{x}e^{t}e^{-t}\mu
_{1}(dt)+\int\limits_{x_{0}}^{x}e^{t}t^{-1}e^{-t}\mu _{1}(dt)\right]
+e^{x}(c_{1}+c_{2}x)e^{1} \\
&=&e^{x}\left[ x^{2}-xx_{0}+\log x-\log x_{0}\right]
+e^{x}(c_{1}+c_{2}x)e^{1} \\
&=&e^{x}\left[ x(\log x+1)-xx_{0}-\log x_{0}+(c_{1}+c_{2}x)e^{1}\right]  \\
&=&e^{x}\left[ x(\log x-1)+(c_{1}e^{1}-\log x_{0})+(c_{2}e^{1}-x_{0}+2)x%
\right]  \\
&=&e^{x}\left[ x(\log x-1)+C_{1}+C_{2}x\right] ,
\end{eqnarray*}%
which is Equation (\ref{ExEulergensol1}), for arbitrary $C_{1},C_{2}\in \Re ,
$as anticipated.\newline
Now using Equation (\ref{a_nrelation2sols}) we find the functions $a_{n}$ as%
\begin{equation*}
a_{n}(x)=\frac{1}{x^{n}}w_{n,2}(x,x_{0},\mathbf{z}_{2})=\frac{e^{x}}{x^{n}}%
\left[ x\int\limits_{x_{0}}^{x}t^{n}e^{-t}\mu
_{1}(dt)+\int\limits_{x_{0}}^{x}t^{n-1}e^{-t}\mu _{1}(dt)\right] +\frac{e^{x}%
}{x^{n}}(c_{1}+c_{2}x),
\end{equation*}%
with the general solution (following the steps above) given by the MTMF%
\begin{equation*}
T_{x,\mathbf{a}(x)}(\mathbb{N})=\sum\limits_{n=0}^{+\infty }\frac{1}{n!}%
w_{n,2}(x,x_{0},\mathbf{z}_{2})=e^{x}\left[ x(\log x-1)+C_{1}+C_{2}x\right] .
\end{equation*}

\item Analytic functions: For $B=\mathbb{N},$ $g(x)=x-x_{o},$ and $%
a_{n}(x)=f^{(n)}(x_{0}),$ $n\in \mathbb{N},$ the resulting MTMF is the
Taylor series of an analytic function $f$ at the point $x_{0}$, given by%
\begin{equation}
f(x)=T_{x-x_{0},\mathbf{a}(x)}(\mathbb{N})=\sum\limits_{n=0}^{+\infty
}f^{(n)}(x_{0})\frac{(x-x_{0})^{n}}{n!}\in \mathcal{A}_{\mathbb{M}}^{\mathbb{%
M}}.
\end{equation}%
Now if a given function $h$ is analytic at $x_{0},$ we have the MTMF%
\begin{equation}
h(x)=T_{x-x_{0},\mathbf{h}(x)}(\mathbb{N})=\sum\limits_{n=0}^{+\infty
}h^{(n)}(x_{0})\frac{(x-x_{0})^{n}}{n!}\in \mathcal{A}_{\mathbb{M}}^{\mathbb{%
M}},  \label{hfunc1}
\end{equation}%
with $\mathbf{h}(x)=[h(x_{0}),$ $h^{(1)}(x_{0}),...],$ where $%
h_{n}(x)=h^{(n)}(x_{0}),$ and $h_{g}(x)=x-x_{0}.$\newline
Using equation (\ref{GenSol2}), and noting that $\left\vert
z_{i}(x)\right\vert <+\infty ,$ for all $i=1,2,...,m<+\infty ,$ we have that%
\begin{gather*}
\left\vert T_{g(x),\mathbf{a}(x)}(\mathbb{N})\right\vert
<\sum\limits_{n=0}^{+\infty }\frac{1}{n!}\left[ \left\vert
z_{0,n}(x)\right\vert +\sum\limits_{i=1}^{m}\left\vert c_{i}\right\vert
\left\vert z_{i}(x)\right\vert \right] <e\sum\limits_{i=1}^{m}\left\vert
c_{i}\right\vert \left\vert z_{i}(x)\right\vert  \\
+\sum\limits_{n=0}^{+\infty }\frac{1}{n!}\sum\limits_{k=1}^{m}\left\vert
z_{k}(x)\right\vert \int\limits_{x_{0}}^{x}\left\vert \frac{W_{k}(\mathbf{z}%
_{m})(t)}{W(\mathbf{z}_{m})(t)}\frac{h^{(n)}(x_{0})}{p_{m}(t)}%
(t-x_{0})^{n}\right\vert \mu _{1}(dt) \\
<e\sum\limits_{i=1}^{m}\left\vert c_{i}\right\vert \left\vert
z_{i}(x)\right\vert +\sum\limits_{k=1}^{m}\left\vert z_{k}(x)\right\vert
\int\limits_{x_{0}}^{x}\left\vert \frac{W_{k}(\mathbf{z}_{m})(t)}{p_{m}(t)W(%
\mathbf{z}_{m})(t)}\right\vert \left\vert h(t)\right\vert \mu
_{1}(dt)<+\infty ,
\end{gather*}%
provided that the integrability conditions (\ref{Integrable1}) hold$,$ so
that $T_{g(x),\mathbf{a}(x)}(\mathbb{N})\in \mathcal{A}_{\mathbb{M}}^{%
\mathbb{M}}.$ Therefore, even in this case where we have a countable number
of differential equations (\ref{MTMFODEs}), Theorem \ref{SolODEsThm} is
directly applicable and can provide us with the solution to the problem (\ref%
{LODE1}).\newline
Alternatively, one can define $B_{N}=\{0,1,...,N\},$ apply Theorem \ref%
{SolODEsThm} to obtain $T_{x-x_{0},\mathbf{a}(x)}$ $(B_{N})$, and then
sending $N\rightarrow \infty $ in equation (\ref{GenSol2}), we have the
general solution%
\begin{equation}
T_{g(x),\mathbf{a}(x)}(B_{N})\rightarrow T_{g(x),\mathbf{a}(x)}(\mathbb{N}),
\end{equation}%
where%
\begin{gather*}
T_{g(x),\mathbf{a}(x)}(\mathbb{N})=\sum\limits_{n=0}^{+\infty }\frac{z_{n}(x)%
}{n!}=\sum\limits_{n=0}^{+\infty }\frac{z_{0,n}(x)}{n!}+e\sum%
\limits_{i=1}^{m}c_{i}z_{i}(x) \\
=\sum\limits_{k=1}^{m}z_{k}(x)\int\limits_{x_{0}}^{x}\frac{W_{k}(\mathbf{z}%
_{m})(t)}{p_{m}(t)W(\mathbf{z}_{m})(t)}\sum\limits_{n=0}^{+\infty }\frac{%
h^{(n)}(x_{0})}{n!}(t-x_{0})^{n}\mu
_{1}(dt)+e\sum\limits_{i=1}^{m}c_{i}z_{i}(x),
\end{gather*}%
and therefore using (\ref{hfunc1}), we have%
\begin{equation*}
T_{g(x),\mathbf{a}(x)}(\mathbb{N})=\sum\limits_{k=1}^{m}z_{k}(x)\int%
\limits_{x_{0}}^{x}\frac{h(t)W_{k}(\mathbf{z}_{m})(t)}{p_{m}(t)W(\mathbf{z}%
_{m})(t)}\mu _{1}(dt)+e\sum\limits_{i=1}^{m}c_{i}z_{i}(x)<+\infty .
\end{equation*}%
In particular, when the integrability conditions (\ref{Integrable1}) hold,
owing to linearity of $L_{m}$ and linearity of the solutions $T_{x-x_{0},%
\mathbf{a}(x)}(B_{N})$ of equation (\ref{GenSol2}), the only other major
requirement on the functions involved ($f$ and $h$), is for them to be
analytic, so that the infinite sums converge.

\item Space-time ODEs: Let $h:\Re \times \mathbb{M}\rightarrow \mathbb{M}$,
a known measurable function in $(\Re \times \mathbb{M},$ $\mathcal{B}(\Re
\times \mathbb{M})),$ with $h(t,x)=T_{h_{g}(x),\mathbf{h}(t)}(B),$ for a
given $B\in \mathcal{B}(\mathbb{N}),$ where $\mathbf{h}(t)=[h_{0}(t),$ $%
h_{1}(t),$ $...]$ $\in $ $\mathbb{M}^{\infty },$ so that%
\begin{equation*}
h(t,x)=\sum\limits_{n\in B}h_{n}(t)\frac{h_{g}(x)^{n}}{n!}.
\end{equation*}%
Consider the general MTMF $f^{B}(t,x)=T_{g(x),\mathbf{a}_{t}}(B)$ of
equation (\ref{DiffTMF}). Then the operator (\ref{LinOperatornthorder})
becomes%
\begin{equation*}
L_{m}\left[ f^{B}\right] =L_{m}\left[ \sum\limits_{n\in B}a_{n}(t)\frac{%
g(x)^{n}}{n!}\right] =\sum\limits_{n\in B}\frac{1}{n!}a_{n}(t)L_{m}\left[
g(x)^{n}\right] ,
\end{equation*}%
and the corresponding problem (\ref{LODE1}) becomes
\begin{equation*}
\sum\limits_{n\in B}\frac{1}{n!}a_{n}(t)L_{m}\left[ g(x)^{n}\right]
=\sum\limits_{n\in B}\frac{1}{n!}h_{n}(t)h_{g}(x)^{n},
\end{equation*}%
so that the MTMF ODEs are given by%
\begin{equation*}
L_{m}\left[ g(x)^{n}\right] =h_{g}(x)^{n},
\end{equation*}%
for all $n\in B\in \mathcal{B}(\mathbb{N}).$ Then the general solution is
given by Theorem \ref{SolODEsThm} (keeping $t$ fixed), with $\mathbf{a}(t,x)$
free of $x$ (a constant sequence that depends on $t$). The variable $t$ is
thought of as time, while $x$ is treated as indicating a spatial location
(in 1-d when $\mathbb{M}=\Re $ and in 2-d when $\mathbb{M}=\mathbb{C}$).

\item Partial Differential Equations with respect to time: In view of the
latter space-time MTMF, we make a modification to the function $h$ and look
for the solution to the partial differential equation (PDE)%
\begin{equation*}
L_{m}\left[ f^{B}\right] =\frac{\partial h(t,x)}{\partial t}%
=\sum\limits_{n\in B}\frac{1}{n!}\frac{dh_{n}(t)}{dt}h_{g}(x)^{n}.
\end{equation*}%
Then the general solution is given once again by Theorem \ref{SolODEsThm},
with $h_{n}$ in the previous example replaced by $\frac{dh_{n}}{dt}$.
\end{enumerate}
\end{example}

\subsection{First order non-linear differential equations and MTMFs}

Consider a modification to the LODE (\ref{LODE1}), where we introduce
non-linearity in the $h$ function; we define the non-linear ordinary
differential equation (NLODE) problem by%
\begin{equation}
y^{\prime }=h(x,y),\text{ under the condition }y(0)=0,  \label{NLODE}
\end{equation}%
where $h$ is a real valued, smooth function with respect to its arguments.
Following \cite{steinberg1984lie}, a Lie series is the exponential of a
differential operator $L$ given by%
\begin{equation}
e^{tL}=\sum\limits_{n=0}^{+\infty }\frac{t^{n}}{n!}L^{n},
\label{LieOperator}
\end{equation}%
for $x\in \Re .$ Let $L=f(y)D,$ where $D=d/dy$, and $f$ some analytic
function near $0.$ Then applying (\ref{LieOperator}) to an analytic function
$g$ about $0$, yields the Lie series%
\begin{eqnarray*}
e^{tL}(g(y)) &=&\sum\limits_{n=0}^{+\infty }\frac{t^{n}}{n!}%
L^{n}(g(y))=\sum\limits_{n=0}^{+\infty }\frac{t^{n}}{n!}\left( f(y)D\right)
^{n}(g(y)) \\
&=&g(y)+f(y)\left( Dg(y)\right) +\frac{1}{2}f(y)\left( f(y)Dg(y)\right) +...,
\end{eqnarray*}%
so that a Taylor series is a special case%
\begin{equation*}
e^{tD}(g(y))=\sum\limits_{n=0}^{+\infty }\frac{t^{n}}{n!}\frac{d^{n}}{dy^{n}}%
(g(y))=g(y+t).
\end{equation*}%
Letting%
\begin{equation*}
L=\left( \frac{\partial }{\partial x}+h(x,y)\frac{\partial }{\partial y}%
\right) ,
\end{equation*}%
the solution of the problem (\ref{NLODE}), has the Lie series representation
(see \cite{zwillinger2021handbook}, eq. (150.7))%
\begin{equation*}
y(x)=\sum\limits_{n=1}^{+\infty }\frac{x^{n}}{n!}\left[ \left( \frac{%
\partial }{\partial x}+h(x,z)\frac{\partial }{\partial z}\right) ^{n}(z)%
\right] |_{z=0}.
\end{equation*}%
This is a MTMF with $B=\mathbb{N}^{+},$ $y=f^{\mathbb{N}^{+}}(x)=T_{x,%
\mathbf{a}^{h}}(\mathbb{N}^{+})\in \mathcal{F}_{\mathbb{M}}^{\mathbb{M}},$
and $\mathbf{a}^{h}=[0,$ $a_{1}^{h}(x),$ $a_{2}^{h}(x),...],$ where%
\begin{equation*}
a_{n}^{f}(x)=\left[ \left( \frac{\partial }{\partial x}+h(x,z)\frac{\partial
}{\partial z}\right) ^{n}z\right] |_{z=0},
\end{equation*}%
$n\in \mathbb{N}^{+}$.

In general, any Lie series defined by (\ref{LieOperator}) is a special case
of MTMFs, since for any operator $L$ we have%
\begin{equation}
e^{c(t)L}(g(y))=\sum\limits_{n=0}^{+\infty }\frac{c(t)^{n}}{n!}%
L^{n}(g(y))=\sum\limits_{n=0}^{+\infty }a_{n}(y)\frac{c(t)^{n}}{n!}=T_{c(t),%
\mathbf{a}}(\mathbb{N}),  \label{LieSeriesMTMF1}
\end{equation}%
where $c$ an analytic function, and $\mathbf{a}=[a_{0}(y),$ $a_{1}(y),$ $%
a_{2}(y),...],$ with $a_{n}(y)=L^{n}(g(y)),$ $n\in \mathbb{N}$.

\section{Concluding Remarks}

Motivated by Taylor's theorem, we introduced and studied properties of a
space of measurable functions that generalizes analytic functions to the
space of multivariate Taylor measure functions $\mathcal{F}_{\Re }^{\Re
^{p}} $. We proposed a new inner product $\rho _{\mathcal{F}}\left(
.,.\right) $ which allowed us to prove that the MTMF space is Hilbert and
Polish. The space emerged as a unifying framework, containing as special
cases many important functions in mathematics. We discussed some classic
function spaces and their connections to $\mathcal{F}_{\Re }^{\Re ^{p}},$
such as collections of analytic functions, continuous functions or $\mathcal{%
L}^{p}$-spaces. These spaces can be thought of as special cases of the
general MTMF space $\mathcal{F}_{\Re }^{\Re ^{p}}$ of Definition (\ref%
{MTMFDef}), with additional integrability conditions when required.

We presented first applications of the proposed MTMF space, in particular
for ODEs, PDEs and first order NLODEs, where a first connection was made
with the Lie series. Additional applications of MTMFs include the
construction of basis functions in Hilbert function spaces, representation
of solutions of PDEs, in particular space-time PDEs, and investigation of
MTMFs and Lie theory. These results will be presented elsewhere.

\section*{Acknowledgements}

I am grateful to Professor Petros Valettas, Department of Mathematics,
University of Missouri, for his suggestion regarding the end of the proof of
Theorem \ref{TaylorPolishMTMF}.

\section*{Conflict of interest}

The author declares no conflict of interest.

\section*{Funding}

This research received no external funding.

\bibliographystyle{plainnat}
\bibliography{MTMF}

\end{document}